\documentclass[12pt]{amsart}

\usepackage{amsmath,amssymb,amscd}
\usepackage[mathscr]{eucal}
\usepackage{multicol}
\usepackage{graphicx}
\usepackage{color}

\newtheorem{defn}{Definition}[section]
\newtheorem{thm}{Theorem}
\newtheorem{cor}{Corollary}[section]
\newtheorem{prop}{Proposition}[section]
\newtheorem{lem}{Lemma}[section]
\newtheorem{clm}{Claim}[section]
\newtheorem{comp}{Complement}[section]

\newcommand{\R}{{\mathbb{R}}}
\newcommand{\Z}{{\mathbb{Z}}}
\newcommand{\TT}{{\mathscr{T}\hspace{.15mm}\mathscr{T}}}
\newcommand{\CokJ}{{\mathrm{Cok}\;J}}
\newcommand{\tinycirc}{{^{{\,}_{\mathrm{o}}}}}

\begin{document}
\title{Tangential Thickness of Manifolds}
\author{S{\l}awomir Kwasik}
\address{Department of Mathematics\\
Tulane University\\
New Orleans, LA 70118}
\email{kwasik@math.tulane.edu}
\thanks{The first author has been supported by Louisiana BOR award
LEQSF(2008-2011)-RD-A-24 and Simons Foundation Grant 281810.}
\author{Reinhard Schultz}
\address{Department of Mathematics\\
University of California, Riverside\\
Riverside, CA 92521}
\email{schultz@math.ucr.edu}
\thanks{{\it AMS $2010$ Subject classification\/}:\quad
Primary 55P10, 55R40, 57R63 ,57Q50.}

\maketitle

\centerline{\it To our wives: Hanna and Deborah}

\begin{abstract}
A notion of tangential thickness of a manifold is introduced.  An
extensive calculation within the class of lens and fake lens
spaces leads to a classification of such manifolds with
thickness 1, 3 or 2$k$, for $k\geq 1$.  On the other hand,
calculations of tangential thickness in terms of the dimension of
the manifold and the rank of the fundamental group show very
interesting and quite surprising correlations between these
invariants.
\end{abstract}

\section{Introduction}

Given two nonhomeomorphic topological spaces, $X$ and $Y$, it is
often interesting and important to specify necessary or sufficient
conditions for $X\times\mathbb{R}$ and $Y\times\mathbb{R}$ to be
homeomorphic, where $\mathbb{R}$ denotes the real line. More
generally, it is also useful to have criteria for determining
whether $X\times\mathbb{R}^k$ and $Y\times\mathbb{R}^k$ are
homeomorphic for some $k\geq 1$ (\textit{cf.} \cite{Bi} or
\cite{La}). If $X$ and $Y$ are closed manifolds the following
result, due to B. Mazur in the smooth and piecewise linear
categories \cite{Maz}, provides an abstract answer; in the
statement of this result below, CAT refers to the category of
smooth, piecewise linear, or topological manifolds, and a
CAT-isomorphism is a diffeomorphism, piecewise linear homomorphism
or homeomorphism, respectively.

\textbf{Stable Equivalence Theorem:} \textit{Let $M$ and $N$ be
closed} CAT-\textit{manifolds. Then $M\times\mathbb{R}^k$ and
$N\times\mathbb{R}^k$ are} CAT-\textit{isomorphic for some
$k\geq 1$ if and only if $M$ and $N$ are tangentially
homotopy equivalent ($\textit{i.e.}$, there is a homotopy
equivalence $f: M\rightarrow N$ such that the pullback of the
stable tangent bundle/microbundle of $N$ is the stable tangent
bundle/microbundle of $M$)}.

\noindent
\textsc{Refinements of the Stable Equivalence Theorem. }
Given two manifolds $M$ and $N$ satisfying the conditions of the
Stable Equivalence Theorem, it is natural to ask the following:

\noindent
\textbf{Optimal Value Question:} \textit{If $M$ and $N$ are
tangentially homotopy equivalent closed} {CAT-}\textit{manifolds,
what is the least value of $k\geq 0$ such that $M\times\R^k$ and
$N\times\R^k$ are} CAT-\textit{isomorphic?}


The Whitney Embedding and Tubular Neighborhood Theorems imply that
$\dim N + 1 = \dim M + 1$ is a universal upper bound
for $k$ in the smooth category. Standard results for piecewise
linear manifolds \cite{HaeW} and results of Kirby-Siebenmann
\cite{KirSie} imply the analog in the piecewise linear and
topological categories, respectively.  However, there are many
pairs of tangentially homotopy equivalent closed manifolds
$\{M,N\}$ such that $M\times \R^k$ and $N\times \R^k$ are CAT-isomorphic
when $k$ is much smaller than the universal upper bound.  In
particular, there are infinite families of pairs of nonhomeomorphic
lens spaces $\{M,N\}$ such that $M\times \R^3$ and $N\times \R^3$ are
diffeomorphic; some examples are described at the end of Section 5.
A related, and instructive, class of examples involves {\bf fake
lens spaces}, which are manifolds that are homotopy equivalent but
not CAT-isomorphic to some lens space $L^{2n-1}$, where $n\geq 3$
and $\pi_1(L^{2n-1})$ is cyclic of odd order.

\begin{thm}
For $L^{2n-1}$ as above, there are infinitely many pairwise
nonhomemorphic smooth manifolds $M^{2n-1}$ such that
$M^{2n-1}\times\R^3$ and $L^{2n-1}\times\R^3$ are diffeomorphic.
\end{thm}

In contrast, the results of \cite{KwaSch3} show that if
$M$ and $N$ are linear space forms such that $M\times\mathbb{R}^2$
is homeomorphic to $N\times\mathbb{R}^2$, then $M$ and $N$ must be
diffeomorphic.

The examples in the theorem are given by surgery
theory \cite{Wall2}.  Specifically, the latter yields homotopy
equivalences $M_\alpha\to L$ which are
smoothly normally cobordant to the identity such that the
manifolds $M_\alpha$ are distinguished by their Atiyah-Singer
invariants as in \cite{BrPeW}, and the $\pi-\pi$
Theorem of \cite{Wall2} implies that the associated maps
$M_\alpha\times\R^3\to L\times \R^3$ are properly homotopic to
diffeomorphisms.  We should also note that techniques of
S. Cappell and J. Shaneson in \cite{CaSh} yield a result analogous
to Theorem 1 for lens spaces with fundamental group
$\Z_{2^r}$ where $r\geq 2$.

\noindent
\textsc{The formal framework. }
One main objective of this paper is to study
the Optimal Value Question for fake lens spaces.  We will
concentrate on the case of (odd) prime order fundamental groups,
although many of our results hold without this restriction.

Precise statements of our results require a concept
we shall call {\bf tangential thickness}.
In order to use the basic techniques of geometric topology, we
need to express our results in terms of
the homotopy structure set $\mathbb{S}^\mathrm{CAT}(M)$ of
a CAT-manifold $M$
({\it e.g.}, see \cite{KirSie} or \cite{Ran2}),
where CAT denotes one of the usual manifold categories
(smooth, piecewise linear or topological).

\textbf{Main Definition. }
Let CAT denote the smooth, piecewise linear or topological
manifold category, let $(N,f)$ be a CAT-homotopy structure
on $M$ ({\it i.e.\/}, $N$ is a closed CAT-manifold and
$f:N\to M$ is a homotopy equivalence), and let $k$ be a
nonnegative integer.  The CAT-homotopy structure
$(N,f)$ is said to have {\bf tangential thickness} $\leq k$
(with respect to CAT) if and only if there is a
$\mathrm{CAT}$-isomorphism
$$F:N\times\mathbb{R}^k~\longrightarrow~M\times\mathbb{R}^k$$
such that $F$ is homotopic to the following composite:
\[
\begin{CD}
N\times \R^k @>\textrm{projection}>> N @>f>>  M
@>\textrm{zero}>\textrm{slice}> M\times \R^k\\
\end{CD}
\]
If two homotopy structures $(N,f)$ and $(N',f')$
define the same element in the homotopy structure set
$\mathbb{S}^\mathrm{CAT}(M)$ of a CAT-manifold,
then there is a CAT-isomorphism $h:N'\to N$ such that
$f'$ is homotopic to $f\tinycirc h$. Moreover, if $(N,f)$ has
tangential thickness $\leq k$ such that $F$ is the CAT-isomorphism
as above, then there is a corresponding
CAT-isomorphism
\[
F'~~=~~F\tinycirc \left(h\times\mathrm{Id}(\R^k)\,\right)
\]
for $(N',f')$ and hence $(N',f')$ also has tangential thickness
$\leq k$ as defined above.  Therefore we can say that an
equivalence class of homotopy structures has tangential thickness
$\leq k$ if some representative satisfies this property, for if
this is true for one representative then it is true for every
other representative.

Following a standard pattern, we shall say that an element
of the homotopy structure set $\mathbb{S}^\mathrm{CAT}(M)$ has
{\bf tangential
thickness equal to} $k$ if it has tangential thickness $\leq k$
but does not have tangential thickness $\leq k-1$.
One then has an increasing sequence of sets
$$
\{M\}~~=~~\TT^{\mathrm{CAT}}_0(M)~~\subseteq~~
\TT^{\mathrm{CAT}}_1(M)~~\subseteq~~\cdots~~\subseteq~~
\TT^{\mathrm{CAT}}(M)
$$
where $\TT^{\mathrm{CAT}}(M)$ is the set of all manifold
structures $(N,f)$ such that $f$ is a tangential homotopy
equivalence and $\TT^{\mathrm{CAT}}_k(M)$ consists of all homotopy
structures with tangential thickness $\leq k$.
As noted earlier, this sequence stabilizes for
$k\geq\dim M+1$; \textit{i.e.}, we have

\begin{center}
$\TT_{\dim M+1}(M)~~=~~\TT_{\dim M+i+1}(M)~~=~~\TT(M)$\qquad\ \ for $i\geq 1$.
\end{center}

\noindent
\textsc{Specialization to lens spaces. }
Frequently we would like to have a version of tangential thickness
which only involves the manifolds $N$ and $M$, with no mention of
a preferred homotopy equivalence $N\to M$.  This can often be done if
$M$ satisfies a reasonably weak rigidity condition (namely, every
tangential homotopy self-equivalence of $M$ is normally cobordant
to the identity).  In Section 5 we show that this applies to
lens spaces whose fundamental groups are isomorphic to $\Z_p$ where
$p$ is an odd prime.  For such examples one can justify the abuses of
language, {\sl the pair $\{M,N\}$ has tangential thickness $\leq k$
or equal to $k$\/}, because the validity of these statements does
not depend upon the choice of a homotopy equivalence $N\to M$
(see Proposition 5.1; note that the homotopy equivalence must be
tangential, for otherwise a homotopy structure cannot have any
tangential thickness).

\noindent
\textsc{Statements of results. }
Our principal results on tangential thickness for fake
lens spaces can now be stated as follows:

\begin{thm}
Let $n\geq 3$, and let $M^{2n-1}$ be a fake lens space (arbitrary
fake spherical space form).  Then $\TT_1^{\mathrm{Top}}(M)$
consists of manifolds h-cobordant to $M$.  These manifolds are
classified by $Wh(\pi_1(M))$ via realization of Whitehead torsion
by h-cobordisms (i.e., the action of $Wh(\pi_1(M))$ on the equivalence
classes of simple homotopy structures for $M$ is free).
\end{thm}

\begin{thm}
Let $M^{2n-1}$, $n\geq 3$, be a fake lens space with $\pi_1(M^{2n-1})\cong
\Z_k$ where $k$ is odd.  Then
$N^{2n-1}$ is in $\TT_2^{\mathrm{Top}}(M)$ if and only if
$N\times\mathbb{R}$ is properly h-cobordant to
$M\times\mathbb{R}$.  The set $\TT_2^{\mathrm{Top}}(M)\smallsetminus
\TT_1^{\mathrm{Top}}(M)$ is in one-to-one correspondence with
a subset of
$\widehat{H}^0 (\Z_2;\widetilde{K}_0(\mathbb{Z}[\pi_1(M)]))$
via the realization of Whitehead torsion by proper $h$-cobordisms.
\end{thm}

\begin{thm}
Let $M^{2n-1}$, $n\geq 3$, be a fake lens space with
$\pi_1(M)\cong\mathbb{Z}_p$, for $p$ an odd prime.  Then the set
$\TT_3^{\mathrm{Top}}(M)$ is the set of homeomorphism classes of
manifolds normally cobordant to $M$.  The set
$\TT_3^{\mathrm{Top}}(M)\smallsetminus\TT_2^{\mathrm{Top}}(M)$ is in
one-to-one correspondence with the nontrivial elements of the free
abelian group
$\mathbb{Z}^{\frac{p-1}{2}}$.  A manifold $N^{2n-1}$ is in
$\TT_3^{\mathrm{Top}}(M)\smallsetminus\TT_2^{\mathrm{Top}}(M)$ if it is
obtained from $M$ by the action of
$\mathbb{Z}^{\frac{p-1}{2}}\subset \widetilde{L}^h_{2n}(\mathbb{Z}_p)$
on the equivalence classes of homotopy structures for $M$ via realization
of elements of the surgery obstruction group by normal cobordism
starting with $M$ (\textit{cf.} \cite{BrPeW}, \cite{Wall2}).
\end{thm}

%
%

{\bf Note. }
The $\rho$-invariant is the invariant denoted by $\sigma$
in \cite{AtSi}.  Standard results on surgery obstruction groups
(\textit{cf.} \cite{Bak}, p. 388) imply that $\rho$ defines
a homomorphism which is injective on the torsion free subgroup
of $L^h_{2n}(\Z_p)$
in the statement of the theorem, the kernel of $\rho$ is the
torsion subgroup of $L^h_{2n}(\Z_p)$, and the latter
is isomorphic to $\widehat{H}^0 \left(\Z_2;\widetilde{K}_0
(\mathbb{Z}[\pi_1(M)])\,\right)$.  By Theorem 3
we know that the action of this torsion subgroup
yields classes in $\TT_2^{\mathrm{Top}}(M)$.

Our general results on $\TT^{\mathrm{Top}}_k(M)$ for $k\geq
3$ are most conveniently stated in terms of the normal invariant
map
$$
\eta: \mathbb{S}^\mathrm{Top}(M)\longrightarrow [M,G/\mathrm{Top}]
$$
(see \cite{Wall2}) in the Sullivan-Wall surgery exact sequence.

\begin{thm}
Let $M^m$ be a compact unbounded topological manifold of dimension
$m\geq 5$, and let $k\geq 3$.  Assume that the image of
the normal invariant map $\eta$ is a subgroup of
$[M,G/\mathrm{Top}]$, where the group operation on the latter is
given by taking direct sums, and also assume that $M$ is weakly
tangentially (topologically) rigid in the sense of Section $5$.  Then
there is an increasing
sequence of subgroups $\theta_k([M,G/\mathrm{Top}])$, defined for
all $k\geq 3$, with the following properties:
\begin{itemize}
\item[(i)] If $k\geq m+1$, then $\theta_k([M,G/\mathrm{Top}])
= \theta_{k+1}([M,G/\mathrm{Top}])$.

\item[(ii)] If $f: N\rightarrow M$ is a (tangential)
homotopy equivalence of
manifolds, then $N\times\R^k$ and $M\times\R^k$ are homeomorphic
by a homeomorphism which is homotopic to the composite
\[
\begin{CD}
M\times \R^k @>\mathrm{projection}>> M @>f>>  N
@>\mathrm{zero}>\mathrm{slice}> N\times \R^k\\
\end{CD}
\]
if and only if $\eta(f)\in\theta_k([M,G/\mathrm{Top}])$.
\end{itemize}
\end{thm}

In view of the first conclusion in this theorem, it is meaningful
to write $\theta([M,G/\mathrm{Top}]) = \theta_k([M,G/\mathrm{Top}])$
if $k\geq m+1$.

\textbf{Remarks.  1.}  We are assuming that the image of $\eta$ is a
subgroup with respect to direct sum in order to avoid possible
problems with the nonadditivity of the surgery obstruction map
$\sigma: [M,G/\mathrm{Top}]\rightarrow L^h_m(\pi_1(M^m),
w_1)$, where the operation on the domain is given by taking
direct sums.  One easy way to ensure that the image of $\eta$ is a
subgroup is to assume that $L^h_m(\pi_1(M^m),w_1) = 0$ so
that $\eta$ must be onto.  This condition holds if $\pi_1(M^m)$
has odd order and $m$ is odd \cite{Wall3}, and therefore Theorem 5
is valid for the examples of primary interest in this paper.

\textbf{2. }  If $M$ is not weakly tangentially rigid and $k\geq 3$, then
Propositions 3.1 and 3.2 yield weaker conclusions about tangential
homotopy equivalences $f:M\to N$ for which the corresponding maps
from $M\times \R^k$ to $N\times R^k$ are homotopic to homeomorphisms.

Theorem 5 implies that the tangential thickness sets
$\TT^{\mathrm{Top}}_k(M)$ are the
inverse images of the normal invariant sets
$\theta_k([M,G/\mathrm{Top}])$
with respect to the normal invariant map $\eta$.  For example,
Theorem 4 translates into the statement
$\theta_3([M,G/\mathrm{Top}])=0$ if $M$ satisfies the hypotheses
in that result.  More generally, this allows us to characterize
the differences between $\TT^{\mathrm{Top}}_k(M)$ and
$\TT^{\mathrm{Top}}_{k-1}(M)$ in terms of the nonzero elements in
the subquotient groups
$\theta_{k+1}([M,G/\mathrm{Top}])/\theta_{k}([M,G/\mathrm{Top}])$.

If $M^{2n-1}$ is a fake lens space with fundamental group $\Z_p$,
where $p$ is an odd prime and $n\geq 3$, then we shall see
that the groups $\theta_k([M,G/\mathrm{Top}])$ are all cyclic
$p$-groups (possibly trivial) and hence the same is true of the
subquotients
$\theta_k([M,G/\mathrm{Top}])/\theta_{k-1}([M,G/\mathrm{Top}])$.
We shall prove that every such subquotient is either trivial or
isomorphic to $\Z_p$.  Furthermore, we shall prove that
$\theta_k([M,G/\mathrm{Top}]) = \theta([M,G/\mathrm{Top}])$
well below the range of the Stable Equivalence Theorem
in part $(i)$ of Theorem 5 (in particular, the equation holds
when $k>{n}/(p-1)\,$), and in about half of the
remaining cases the subquotient is isomorphic to $\Z_p$.  We shall
begin with the cases that are the simplest to describe:

\begin{thm}
Let $p$ be an odd prime, let $n\geq 3$ and let $M^{2n-1}$ be
a fake lens space with fundamental group $\Z_p$.  Assume further
that $n\not\equiv 0$ \textrm{mod} $p-1$.  Then the subquotients
$$\theta_k([M,G/\mathrm{Top}])/\theta_{k-1}([M,G/\mathrm{Top}]),\quad
\theta_{2j+2}([M,G/\mathrm{Top}])/\theta_{2j}([M,G/\mathrm{Top}])
$$
are given as follows:
\begin{itemize}
\item[(i)]
$\theta_{k+1}([M,G/\mathrm{Top}])/\theta_k([M,G/\mathrm{Top}]) =
0$ if $k\geq 2\left[\frac{n}{p-1}\right] +2$, where $[ - ]$
denotes the greatest integer function.

\item[(ii)] If $k=2j$ and $1\leq j\leq
\left[\frac{n}{p-1}\right]$ then
$\theta_{2j+2}([M,G/\mathrm{Top}])/\theta_{2j}([M,G/\mathrm{Top}])
\cong\Z_p$; we set $\theta_2([M,G/\mathrm{Top}])=0$ by definition.

\item[(iii)] If $2\leq j\leq \left[\frac{n}{p-1}\right]$,
then either $\theta_{2j+1}([M,G/\mathrm{Top}]) =
\theta_{2j}([M,G/\mathrm{Top}])$ or else\\
$\theta_{2j+1}([M,G/\mathrm{Top}]) =
\theta_{2j+2}([M,G/\mathrm{Top}])$.
\end{itemize}
\end{thm}

There is a similar but slightly weaker conclusion when $n\equiv 0$
mod $p-1$.

\begin{thm}
Suppose we are in the same setting as in Theorem $6$, but $n\equiv
0$ \textrm{mod} $p-1$.  Then $(i)$ and $(iii)$ remain valid.  However, if
$k=2j$ and
$$1\leq j\leq \left[\frac{n}{p-1}\right]$$
then
$\theta_{k+2}([M,G/\mathrm{Top}])/\theta_{k}([M,G/\mathrm{Top}])
\cong\Z_p$ except for precisely one value $j_0$ of $j$.
\end{thm}

We shall say more about the exceptional value in Section 7;
unfortunately, our methods only yield limited information about
the exceptional value $j_0$, but we shall provide some evidence
for conjecturing that $j_0=1$ in all cases.

Here is a more qualitative consequence of the preceding results:

\begin{thm}
Let $L^{2n-1}$ be a lens space with $n\geq 3$.
\begin{itemize}
\item[(i)] If $n\not\equiv 0$ {\rm mod} $p-1$, then for each $j$ such
that $1\leq j\leq \left[\frac{n}{p-1}\right]$, there exist
manifolds $L_j$ tangentially homotopy equivalent to $L$ such that
$L_j\times\R^{2j}$ and $L\times\R^{2j}$ are not homeomorphic but
$L_j\times\R^{2j+2}$ and $L\times\R^{2j+2}$ are homeomorphic.

\item[(ii)] If $n\equiv 0$ {\rm mod} $p-1$, then the same conclusion
holds for all but one value of $j$ such that $1\leq
j\leq \left[\frac{n}{p-1}\right]$.

\item[(iii)] If $N$ is a fake lens space which is tangentially
homotopy equivalent to $L$ and $k\geq
2\left[\frac{n}{p-1}\right] +2$, then $L\times\R^k$ and $N\times\R^k$
are homeomorphic.
\end{itemize}
\end{thm}

\noindent
\textsc{Outline of this paper. }
The proofs of the results stated above
will appear in Sections 2--7 below.  In
Section 2 we shall use surgery-theoretic methods ({\it cf.}
\cite{KwaSch3}) to prove
Theorems 2 and 3.  Section 3 gives a surgery-theoretic criterion
for two manifolds to have tangential thickness $\leq k$,
where $k\geq 3$;  most of this material is surely well known,
but we include it since it is fundamental to our work and
difficult to extract from literature.  In the case of
odd-dimensional $\Z[\frac{1}{2}]$ homology spheres, these results
will be restated very simply in terms of desuspending classes in
the stable cohomotopy groups of such manifolds (see Proposition 3.4).
We specialize the general setting of
Section 3 to fake lens spaces in Section 5; this uses a variety of
results about the structure of the classifying spaces for surgery
theory (the standard reference being \cite{MadMilJ}).
In Section 5 we shall discuss some facts about lens spaces and
criteria which lead to simpler discussion of tangential thickness
for such objects in most cases; at the end of the section we also
discuss some questions about tangential thickness which deal
exclusively with genuine lens spaces.  Section 6
analyzes the cohomotopy desuspension questions from Section 3 for
the case of $\Z_p$ lens spaces using the work of F. Cohen, J.C.
Moore and J.  Neisendorfer (\textit{e.g}, see \cite{CohMoNei1},
\cite{CohMoNei2}, and \cite{Nei}) on exponents of
homotopy groups. We shall bring everything together in Section 7
to prove Theorems 4--8. Finally, Section 8.1 contains some comments
and remarks concerning smooth tangential thickness.
Some of the techniques and ideas of
this paper were applied in \cite{BKS} and \cite{BKS1} when studying and
classifying open complete manifolds of nonnegative curvature (see also
\cite{Ott} for further results on such questions).  By
the results of J. Cheeger and D. Gromoll \cite{ChGr}, such manifolds are
diffeomorphic to the total space of a normal bundle to a compact
locally geodesic submanifold called a \textbf{soul}. An obvious variation
on the notion of the Optimal Value Question in this case leads to
a notion of
\textsl{twisted} tangential thickness. The twisted tangential thickness and
a sample of applications of our techniques to the topology of nonnegatively curved
manifolds are briefly discussed in Section 8.2.

The methods and techniques employed in this paper are a mixture
of geometric and algebraic
considerations involving $K$--theory, surgery and homotopy theory.
Perhaps the main novelty in the paper is the  study of tangential
normal maps represented by $\{M, S^0\}$, by finding the least
$r$ such that a given class desuspends to $[S^r M,S^r]$ ({\it i.e.},
the sphere of origin for the class).
Such homotopy--theoretic problems are often important, interesting
and challenging, and they have been studied extensively from many
different viewpoints (c.f. \cite{Nei}, \cite{GroZi}, \cite{Ed}). It
seems likely that such approaches can yield applications to a variety
of questions involving classification of manifolds.

\noindent
\textsc{Acknowledgment. }
We are grateful to the referee for numerous comments and suggestions
which have improved this paper in several ways.

\section{Results in Low Codimensions}

In this section and the next, we shall derive the basic surgery
theoretic conditions for determining the tangential homotopy
equivalences $h: M\rightarrow N$ such that
$h\times\mathrm{Id}_{\R^k}$ is homotopic to a
homeomorphism.  As in many other situations within geometric
topology, the cases with codimension $k\geq 3$ differ greatly
from the cases where $k=1$ or 2, and in this section we shall
dispose of the latter cases.

\begin{proof}{\sf (Theorem 2)}\quad
Let $M^{2n-1}~(n\geq 3)$ be a fake spherical space form, and let
$f: N^{2n-1}\rightarrow M^{2n-1}$ a tangential homotopy
equivalence.  Suppose $N\times\mathbb{R}$ and $M\times\mathbb{R}$
are isomorphic.  Then it follows that $N$ and $M$ are
$h$-cobordant.  The action of the Whitehead group $Wh(\pi_1(M))$
on the equivalence classes of simple homotopy structures for
$M$ is free by the main result of \cite{KwaSch2}.

On the other hand, if $(W;N, M)$ is an $h$-cobordism between $N$
and $M$, then $W\times S^1$ is an $s$-cobordism between $N\times
S^1$ and $M\times S^1$. Thus $M\times S^1$ is isomorphic to
$N\times S^1$ and hence $M\times\mathbb{R}$ and
$N\times\mathbb{R}$ are isomorphic as well.
\end{proof}

\begin{proof}{\sf (Theorem 3)}\quad
Let $\pi\cong\pi_1(M^{2n-1})$ be the fundamental group of
$M^{2n-1}$. If $N^{2n-1}\in\TT_2(M)$, then $N$ is a fake lens
space and there exists a homeomorphism
$h:N\times\mathbb{R}^2\rightarrow M\times\mathbb{R}^2$. This
yields an $h$-cobordism $W$ between $N\times S^1$ and $M\times
S^1$ (\textit{cf.} \cite{KwaSch3}).  By taking infinite cyclic
coverings, one gets a proper $h$-cobordism $\widetilde{W}$
between $N\times\mathbb{R}$ and $M\times\mathbb{R}$.

Conversely, if there is a proper $h$-cobordism $V$ between
$N\times\mathbb{R}$ and $M\times\mathbb{R}$, then $V\times S^1$ is
a product cobordism between $N\times\mathbb{R}\times S^1$ and
$M\times\mathbb{R}\times S^1$.  In particular,
$N\times\mathbb{R}\times S^1\approx M\times\mathbb{R}\times S^1$
and hence $N\times\mathbb{R}\times\mathbb{R}\approx
M\times\mathbb{R}\times\mathbb{R}$ (\textit{i.e.},
$N\times\mathbb{R}^2\approx M\times\mathbb{R}^2$).

Now, let $\tau_0\in Wh(\widetilde{W}, M\times\mathbb{R})\cong
\widetilde{K}_0(\mathbb{Z}[\pi])$ (\textit{cf.} \cite{Sie1}) be a
proper Whitehead torsion of this proper $h$-cobordism. In analogy with
the compact case (\textit{cf.} \cite{Coh}) there is an
involution on $Wh(\widetilde{W})$ and a duality between $\tau_0\in
Wh(\widetilde{W}, M\times\mathbb{R})$ and
the Whitehead torsion of the inclusion of other end $\tau_1\in
Wh(\widetilde{W}, N\times\mathbb{R})$, and it is given by $\tau_1 =
(-1)^{\dim (M\times\mathbb{R})}\tau_0^*$. Hence $\tau_1 =
\tau_0^*$.

Let $f: N\times\mathbb{R}\rightarrow M\times\mathbb{R}$ be a
proper homotopy equivalence given by the composition of the
inclusion $i$ and retraction $r$:
\[
\begin{CD}
N\times\mathbb{R} @>i>> \widetilde{W} @>r>> M\times\mathbb{R}
\end{CD}
\]
It follows that $\tau(f) = \tau_0 - \tau_1 =
\tau_0 - \tau_0^*$.  However, $f$ is
properly homotopic to a map $f_0\times\textrm{Id}_{\mathbb{R}}:
N\times\mathbb{R}\rightarrow M\times\mathbb{R}$ (\textit{cf.}
\cite{Wall1}, Lemma 2, p. 61), with $f_0:N\rightarrow M$. In particular, as
$f_0\times\textrm{Id}_{S^1}: N\times S^1\rightarrow M\times S^1$
is a simple homotopy equivalence (\textit{cf.} \cite{Coh}), so
must be $f_0\times \textrm{Id}_{\mathbb{R}}$.  As a consequence,
$\tau_0 = \tau_0^*$ and $f: N\times\mathbb{R}\rightarrow
M\times\mathbb{R}$ is a proper simple homotopy equivalence.

The standard construction shows that elements in
$Wh(\widetilde{W})$ of the form $\rho + \rho^*$ can be realized by
an inertial proper $h$-cobordism.  Consider

\begin{center}
$\widehat{H}^0\left(\Z_2;\widetilde{K}_0(\mathbb{Z}[\pi])\,\right)~~=~~
\left\{\displaystyle\frac{\tau = \tau^*}{\tau + \tau^*}\right\}$~.
\end{center}

\begin{clm}
Realization of elements in
$\widehat{H}^0(\Z_2;\widetilde{K}_0(\mathbb{Z}[\pi]))$ via proper
h-cobordisms starting with $M\times\mathbb{R}$ yields manifolds of
the form $N\times\mathbb{R}$ on the other end.
\end{clm}

\begin{proof}
To see this, let $(\overline{W}; M\times\mathbb{R}, K)$ be a
proper $h$-cobordism with $\tau_0\in Wh(\overline{W},
M\times\mathbb{R})$,
$\tau_0\in\widehat{H}^0(\Z_2;\widetilde{K}_0(\mathbb{Z}[\pi]))$.  Then
there is a proper homotopy equivalence
$$f: K\hookrightarrow
\overline{W} \rightarrow M\times\mathbb{R}$$
which is simple.  By
the one-sided splitting theorem for proper maps and noncompact
manifolds (\textit{cf.} \cite{Tay}), $f$ is properly homotopic to
a map $g$ with $g^{-1}(M\times\{0\}) = N\subset K$ and $g|_N :
N\rightarrow M\approx M\times\{0\}$ a homotopy equivalence.  We
have a splitting of $g$ into $g|K_0$ and $g|K_1$ where
$g|{K_0} : K_0\rightarrow M\times [0,\infty)$ and
$g|{K_1} : K_1\rightarrow (-\infty, 0]\times M$ are
proper homotopy equivalences.  Now, the Collaring Theorem of
Siebenmann (\textit{cf.} \cite{Sie2}) implies $K_0\approx N\times
[0,\infty)$ and $K_1\approx (-\infty,0]\times N$, and hence
$K\approx N\times\mathbb{R}$.
\end{proof}

\textbf{Remark. }
In order to get more information about the action of
$\widehat{H}^0(\Z_2;\widetilde{K}_0(\mathbb{Z}[\pi]))$ and
manifold classes in $\TT_2^{\mathrm{Top}}(M)\smallsetminus
\TT_1^{\mathrm{Top}}(M)$, one can use the proper surgery
theory of S. Maumary and L. Taylor (see \cite{Maum}, \cite{Maum1},
\cite{Tay} and \cite{PedRan}). Consider the analog of the
Sullivan-Wall long exact sequence for proper surgery theory:

\begin{center}
$\cdots\rightarrow L^{s,\mathrm{open}}_{\ast +1}(M\times\R)
\rightarrow \mathbb{S}^{s,\,{\mathrm{Top}}}(M\times\R) \rightarrow
[M\times\R; G/\mathrm{Top}] \rightarrow
L^{s,\mathrm{open}}_{\ast}(M\times\R) \rightarrow \cdots$
\end{center}

We have $L^{s,\mathrm{open}}_{\ast +1}(M\times\mathbb{R})\cong
L^h_{even}(\pi)\cong L^{p,s}_{even}(\pi)\oplus
\widehat{H}^0(\Z_2;\widetilde{K}_0(\mathbb{Z}[\pi]))$ (\textit{cf.}
\cite{Bak}). In order to describe the set
$\TT_2^{\mathrm{Top}}(M)\smallsetminus
\TT_1^{\mathrm{Top}}(M)$ precisely one must face the well known and
in general difficult problem of deciding which elements in
$\mathbb{S}^{c,\mathrm{Top}}(M\times\R)$ are represented by proper
homotopy self-equivalences.
\end{proof}


\section{Tangential Thickness and Normal Invariants}

Suppose that $M$ and $N$ are closed $n$-manifolds and $k\geq
3$ is such that $n+k\geq 6$.  Surgery theory then yields the
following criteria for $M\times\R^k$ and $N\times\R^k$ to be
homeomorphic:

\begin{prop}
If $M$, $N$ and $k$ are as above, the $M\times\R^k$ is
homeomorphic to $N\times\R^k$ if and only if the compact bounded
manifolds $M\times D^k$ and $N\times D^k$ are $h$-cobordant in the
following sense:  There is a compact manfold with boundary
$X^{n+k+1}$ and a compact manifold with boundary $W^{n+k}\subseteq
\partial X$ such that the following hold:
\begin{itemize}
\item[(i)] $\partial W^{n+k}$ is homeomorphic to a disjoint union
of $M\times S^{k-1}$ and $N\times S^{k-1}$

\item[(ii)] $\partial X^{n+k+1}\cong M\times D^k\cup W\cup N\times
D^k$, where $M\times D^k\cap W = M\times S^{k-1}$ and\\ $N\times
D^k\cap W = N\times S^{k-1}$

\item[(iii)] The inclusion of pairs $(M\times D^k, M\times
S^{k-1})\subseteq (\partial X, W)\subseteq (X, W)$ and\\ $(N\times
D^k, N\times S^{k-1})\subseteq (\partial X, W)\subseteq (X, W)$
are homotopy equivalences of pairs.
\end{itemize}

More precisely, if $M$, $N$ and $k$ are as above and
$f: N\rightarrow M$ is a homotopy equivalence of
manifolds, then $N\times\R^k$ and $M\times\R^k$ are homeomorphic
by a homeomorphism which is homotopic to the composite
\[
\begin{CD}
M\times \R^k @>\mathrm{projection}>> M @>f>>  N
@>\mathrm{zero}>\mathrm{slice}> N\times \R^k\\
\end{CD}
\]
if and only if $M\times D^k$ is topologically $h$-cobordant
to $N\times D^k$ by a map homotopic to
\[
\begin{CD}
M\times D^k @>\textrm{projection}>> M @>f>>  N
@>\textrm{zero}>\textrm{slice}> N\times D^k~.\\
\end{CD}
\]
\end{prop}

{\sl Notation. }  If $(V_0,W,V_1)$ is an $h$-cobordism, we often say
that $V_0$ and $V_1$ are $h$-cobordant by the map $V_0\to W\to V_1$,
where $V_0\to W$ is the inclusion of $V_0$ in $W$, and $W\to V_1$ is a
homotopy inverse to the inclusion of $V_1$ in $W$.

\begin{proof}
This is fairly standard.  If $M\times\R^k$ and $N\times\R^k$ are
homeomorphic, then the homeomorphism maps $M\times D^k$ into some
subset $N\times r\, D^k$, where $r\, D^k$ is the disk of
radius $r$ for some very large value of $r$.  Let $W$ be the
bounded manifold $N\times r\, D^k\smallsetminus\mathrm{Int}(M\times
D^k)$, and take $X$ to be $N\times D^k\times [0,1]$.  The
decomposition of $\partial X$ in $(ii)$ is then given by
identifying $M\times D^k$ with $M\times D^k\times\{0\}$, $W$ with
$W\times\{0\}$ and $N\times D^k$ with $N\times r\,
D^k\times\{1\} \cup N\times\partial(r\, D^k)\times[0,1]$.  It
is then fairly straightforward to check that the inclusions in
$(iii)$ are homotopy equivalences of pairs. Conversely, if we are
given $X$ as in the theorem, then it follows that
$X\smallsetminus\mathrm{Int}(W)$ is a proper $h$-cobordism from
$M\times\mathrm{Int}(D^k)\cong M\times\R^k$ to
$N\times\mathrm{Int}(D^k)\cong N\times\R^k$ in the sense of
\cite{Sie1}. Then by the proper $h$-cobordism theorem of
\cite{Sie1}, it follows that $M\times\R^k$ and $N\times\R^k$ are
homeomorphic because the proper Whitehead group is trivial in
this case (see the theorems on pages 483--484 of \cite{Sie1} and
observe that if $k\geq 3$ then the fundamental group for the end of
$X\times\R^k$ maps isomorphically to the fundamental group of $X$
if $X$ is a finite complex).

Finally, the more precise formulation at the end of the proposition
is an immediate consequence of the construction, for in both
cases the latter is defined by restricting the data in the hypothesis
to certain subsets.
\end{proof}

\begin{comp}
Similar results are true in the categories of piecewise linear
{\rm (PL)} or smooth manifolds if we stipulate that all manifolds
lie in the given category and the homeomorphisms are
{\rm PL}-homeomorphisms or diffeomorphisms, respectively.
\end{comp}

This is true because one has analogs of the proper $h$-cobordism
theorem in the PL and smooth categories (in fact, they predate the
topological version).  In the smooth category there are some
issues about rounding corners in a product of two bounded smooth
manifolds, but there are standard ways of addressing such points.
({\it e.g.}, see Section I.3 of \cite{Co}, the appendix to \cite{BoSe}
or \cite{joyce}).

These results lead to the use of surgery theoretic structure sets;
the latter are defined for closed manifolds in \cite{Ran2} and one
can treat the bounded case using maps and homotopy equivalences of
pairs as in Chapter 10 of Wall's book \cite{Wall2}.  In order to
translate Proposition 3.1 and Complement 3.1 into the language of
structure sets, we need to work with certain function spaces.
Following James \cite{James}, we shall denote the identity component
of the continuous function space $\mathscr{F}(S^{k-1},S^{k-1})$ by
$SG_k$, and $SF_{k-1}$ will denote the subspace of basepoint
preserving maps (which is also arcwise connected). By the results
of \cite{James} and \cite{Stasheff}, there is a
Serre fibration $SF_{k-1}\rightarrow SG_k\rightarrow S^{k-1}$ and
a corresponding classifying space fibration $S^{k-1}\rightarrow
BSF_{k-1}\rightarrow BSG_k$.  The space of degree zero basepoint
preserving self-maps is homeomorphic to the component
$\Omega^{k-1}_0 S^{k-1}$ of the constant map iterated loop space
$\Omega^{k-1} S^{k-1}$, and the map $w: \Omega^{k-1} S^{k-1}
\rightarrow SF_{k-1}$ sending $f: S^{k-1}\rightarrow S^{k-1}$ to
the composite
\[
\begin{CD}
S^{k-1} @> \mathrm{pinch} >> S^{k-1}\vee S^{k-1} @> f\vee\mathrm{Id} >>
S^{k-1}\vee S^{k-1} @> \mathrm{fold} >> S^{k-1}\\
\end{CD}
\]

\noindent
is a homotopy equivalence.  It is important to recognize that this
homotopy equivalence \textbf{does not send} the loop sum on $\Omega^{k-1}
S^{k-1}$ to the composition product in $SF_{k-1}$ (the precise relationship
is described at the beginning of Section 6).  The unreduced
suspension functor defines
continuous homomorphisms $SG_k\rightarrow SF_{k+1}$, and if
$\Omega^{k-1} S^{k-1}\rightarrow \Omega^k S^k$ is the suspension
map induced by the suspension adjoint $\sigma : S^{k-1}\rightarrow
\Omega S^k$, then we have the following homotopy commutative
diagram:

\[
\begin{CD}
\Omega^{k-1}_0 S^{k-1} @>=>> \Omega^{k-1}_0 S^{k-1}
@> \Omega^{k-1}\sigma >> \Omega^k_0 S^k\\
@ V{w_{k-1}}VV @VVV @VV{w_k}V\\
SF_{k-1} @>>> SG_k @>>> SF_k\\
\end{CD}
\]
{}

\centerline{{\bf (~Diagram 3.0~)}}

The preceding chain of maps can be extended by adjoining
$SG_{k+1}$ on the right, and if we take limits as
$k\to\infty$ via suspensions we obtain the homotopically
equivalent
topological monoids $SG$ and $SF$.
With this preparation, we can restate Proposition 3.1 and Complement
3.1 in the piecewise linear and topological categories as follows:

\begin{prop}
Let $M$ and $N$ be closed connected \textrm{PL} (resp., topological)
manifolds with dim $M$ = dim $N\geq 5$, let $k\geq 3$, and let
$f: M\rightarrow N$ be a homotopy equivalence.  Then $M\times
D^k$ is piecewise linearly (resp., topologically) $h$-cobordant
to $N\times D^k$ by the canonical homotopy equivalence as above
(inclusion of one boundary pieces followed by retraction onto
the other boundary piece), which is homotopic to
\[
\begin{CD}
M\times D^k @>\textrm{projection}>> M @>f>>  N
@>\textrm{zero}>\textrm{slice}> N\times D^k\\
\end{CD}
\]
if and only if the normal invariant in $[N,G/\mathrm{PL}]$ (resp.,
$[N,G/\mathrm{Top}]$) lies in the image of $[N,SG_k]$ under the map
induced by the composite $G_k\rightarrow G\rightarrow
G/\mathrm{PL}$ (resp., $SG_k\rightarrow SG\rightarrow
SG/\mathrm{Top}$).
\end{prop}

If the homotopy equivalences in the preceding two propositions are
simple homotopy equivalences, then by the $s$-cobordism theorems
in the respective categories (see \cite{KirSie}, p. 4)
one has stronger conclusions:

\begin{itemize}
\item[(i)]  If $M$, $N$ and $k$ are as above and lie in the category
of smooth, piecewise linear or topological manifolds and the
homotopy equivalence from $M$ to $N$ is simple, then $M\times\R^k$ is
(respectively) diffeomorphic, piecewise linearly homeomorphic, or
homeomorphic to $N\times\R^k$ if and only if the compact bounded
manifolds $M\times D^k$ and $N\times D^k$ are diffeomorphic,
piecewise linearly homeomorphic or homeomorphic respectively.

\item[(ii)] If everything lies in either the piecewise linear or
topological category and the homotopy equivalence from $M$ to
$N$ is simple, then $M\times D^k$ is piecewise linearly
homeomorphic (resp., homeomorphic)
to $N\times D^k$ by a map as above if and only if the normal
invariant lifts as in Proposition $3.2$.
\end{itemize}

There is an analog of Proposition 3.2 in the smooth category, but the
proof is longer and and omitted because we do not need the smooth
version of this result.

\begin{proof}
We begin with the case of the PL category since the argument is
simpler but also contains the ideas to be employed in the
topological category.  Given a homotopy equivalence $f:
M\rightarrow N$, we want to consider the homotopy structure on
$N\times D^k$ given by the product map $f\times\mathrm{Id}_{D^k}$.
Standard properties of normal invariants imply that
$\eta(f\times\mathrm{Id}_{D^k})=p^*\eta(f)$, where $\eta(-)$
denotes the normal invariant and $p^*:
[N,G/\mathrm{PL}]\rightarrow [N\times D^k, G/\mathrm{PL}]$ is
induced by the coordinate map $p: N\times D^k\rightarrow N$;
the map $p^*$ is an isomorphism because $D^k$ is contractible.
Since $k\geq 3$, it follows that the maps $\pi_i(N\times S^{k-1})
\rightarrow \pi_i(N\times D^k)$ are isomorphisms for $i=0$ or 1.
Therefore the $\pi-\pi$ theorem of \cite{Wall2} implies that the normal
invariant map $\mathbb{S}(N\times D^k)\rightarrow [N\times
D^k,G/\mathrm{PL}]$ and the corresponding map for simple structures
$\mathbb{S}^s(N\times D^k)\rightarrow [N\times
D^k,G/\mathrm{PL}]$  are both 1--1 and onto.  In fact, the inverse to
the forgetful map from $\mathbb{S}^s(N\times D^k)$ to
$\mathbb{S}(N\times D^k)$ is given geometrically as follows:  Given a
homotopy structure $(W,\partial W)\to (N\times D^k,N\times S^{k-1})$
with Whitehead torsion $\alpha$, take the simple homotopy structure
obtained by attaching an $h$-cobordism with Whitehead torsion $-\alpha$
along $\partial W$, and let $(W',\partial W')$ be the result of this
construction; standard Whitehead torsion formulas then imply that the
associated homotopy equivalence
$(W',\partial W')\to (N\times D^k,N\times S^{k-1})$ is simple.

The proof for the embedding theorem of
Browder, Casson, Haefliger, Sullivan and Wall (see \cite{Rourke}, (8.10),
p. 161), implies that there is a piecewise linear homeomorphism from
$W'$ to the total space of some $k$-dimensional block bundle over $N$
which we shall call $\xi$ (see \cite{RS1}, \cite{RS2},
\cite{RS3} for background on block bundles).  If we denote the
total space of this block bundle by $E(\xi)$, then by the
$h$-cobordism and $s$-cobordism theorems we can retrieve $W$ by
attaching an $h$-cobordism with Whitehead torsion $\alpha$ along
$\partial E(\xi)\cong \partial W'$.  In particular, if the
homotopy structure on $N\times D^k$ is given by crossing a homotopy
equivalence $f:M\to N$ with the identity on $D^k$, we obtain a block
bundle $\xi$ on $N$ and a piecewise linear embedding $F:E(\xi)\to
M\times\mathrm{Int}\,(D^k)$ such that the complement of the image of
$F$ is an $h$-cobordism and the composite
\[
\begin{CD}
N@>z>> E(\xi) @>F>> M\times D^k @>\textrm{projection}>> M @>f>>  N\\
\end{CD}
\]
is homotopic to the identity, where $z$ denotes the zero section
inclusion for a block bundle.

The data in the preceding paragraph correspond
to a unique class $\alpha\in [N,G_k/\widetilde{\mathrm{PL}}_k]$
with the following properties:

\begin{itemize}
\item[(i)] The image of $\alpha$ in $[N,G/\mathrm{PL}]$ under a
canonical stabilization map $G_k/\widetilde{\mathrm{PL}}_k
\rightarrow G/\mathrm{PL}$ (which is a homotopy equivalence) is
the normal invariant $\eta(f)$.

\item[(ii)] The image of $\alpha$ in
$[N,\mathrm{B}\widetilde{\mathrm{PL}}_k]$ under a canonical map
$G_k/\widetilde{\mathrm{PL}}_k \rightarrow
\mathrm{B}\widetilde{\mathrm{PL}}_k$ classifies the block bundle
$\xi$.
\end{itemize}

Basic results on block bundles imply that a block bundle $\xi$
over a manifold $N$ is trivial
if and only if its total space $E(\xi)$ is PL-homeomorphic to
$N\times D^n$ such that an appropriate diagram commutes
(see Section 4 of \cite{RS1}).
Furthermore, by the $s$-cobordism Theorem it follows that
$M\times D^k$ is piecewise linearly $h$-cobordant to
$N\times D^k$ by a map homotopic to
\[
\begin{CD}
M\times D^k @>\textrm{projection}>> M @>f>>  N
@>\textrm{zero}>\textrm{slice}> N\times D^k\\
\end{CD}
\]
if and only if the image of $\alpha$ in
$[N,\mathrm{B}\widetilde{\mathrm{PL}}_k]$ is trivial.  The latter
is true if and only if $\alpha$ lies in the image of the map
$[N,G_k]\rightarrow [N,G_k/\widetilde{\mathrm{PL}}_k]$, and hence
the result follows in the piecewise linear category.  The proof in
the topological category is similar, but one must replace the theory
of piecewise linear block bundles with a corresponding theory of
topological regular neighborhoods
as in \cite{RS4} and \cite{Ed}.  One crucial step in the PL
proof uses the
fact that the stabilization map $G_k/\widetilde{\mathrm{PL}}_k
\rightarrow G/\mathrm{PL}$ is a homotopy equivalence if
$k\geq 3$.  The corresponding fact for the map
$G_k/\widetilde{\mathrm{Top}}_k \rightarrow G/\mathrm{Top}$ is
contained in \cite{RS4}.
\end{proof}

\textbf{Remark.} Since the main objects of interest in this paper
are odd-dimensional $\Z[\frac{1}{2}]$-homology spheres and topological
equivalence coincides with piecewise linear equivalence for such
manifolds by \cite{KirSie}, all we really need in this paper is the
piecewise linear case of the preceding result.

If $X$ is a connected finite complex, then Diagram 3.0
yields an
isomorphism of sets from the stable cohomotopy group $\{X,S^0\}$
to $[X,SG]$.  Under this isomorphism, the image of the map
$[X,SG_k]\rightarrow [X,SG]$ is trapped between the images of
the iterated suspension homomorphisms $[S^{k-1}X,S^{k-1}]\rightarrow
\{X,S^0\}$ and $[S^kX,S^k]\rightarrow \{X,S^0\}$.  The results of
\cite{James} show that the image of $[X,SG_k]\rightarrow [X,SG]$
corresponds
to the image of $[S^kX,S^k]\rightarrow \{X,S^0\}$ if
$\dim X\leq 2k-2$.  We shall also need the following criteria for
determining whether a class in $[X,SG]$ lifts back to $[X,SG_k]$:

\begin{prop}\label{prop33}
Let $X$ be a connected finite complex and let $\alpha\in [X,SG]$
be a class such that $\alpha$ lifts to $[X,SG_3]$; take the group
structures on these spaces induced by the composition products on
the function spaces $\mathscr{F}(S^3,S^3)$ and
$\displaystyle\lim_{m\rightarrow\infty} \mathscr{F}(S^m,S^m)$.
Then $\alpha = \alpha_1+\alpha_2$ where $\alpha_2$ lies in the
image of $[X,SO]\rightarrow [X,SG]$ (where $SO$ is the group
$\displaystyle\lim_{m\rightarrow\infty} SO_m$) and $\alpha_1$
corresponds to an element in the image of $[S^2X,S^2]\rightarrow
\{X,S^0\}$.
\end{prop}

\begin{proof}
It will suffice to show that the images of $[X,SG_3]$ and
$[X,SF_2]$ in $[X,G/O]$ are equal, for this implies that the image
of $[X,SG_3]$ in $[X,SG]$ is generated by $[X,SF_2]$ and $[X,SO]$,
and it follows from Diagram 3.0 that
the image of $[X,SF_2]$ in $[X,SG]$ corresponds to
the image of $[S^2X,S^2]$ in $\{X,S^0\}$.

We begin with the following commutative diagram whose rows are
given by fibrations:
\[
\begin{CD}
SO_2 @>>> SO_3 @>>> S^2 @>>> BSO_2 @>>> BSO_3\\
@VVV @VVV @| @VVV @VVV\\
SF_2 @>>> SG_3 @>>> S^2 @>>> BSF_2 @>>> BSG_3\\
\end{CD}
\]
It follows that the fibers of $BSO_2\rightarrow BSF_2$ and
$BSO_3\rightarrow BSG_3$, which are $SF_2/SO_2$ and $SG_3/SO_3$,
are homotopy equivalent.  Since the map $SO_3\rightarrow SO$ is
well known to be 2-connected and $SG_3\rightarrow SG$ is also
2-connected by \cite{James}, it follows that
$SG_3/SO_3\rightarrow G/O$ is
2-connected.  Therefore, $\pi_1(SG_3/SO_3)\cong \pi_1(G/O) = 0$,
so that $SF_2/SO_2$ is also simply connected.  Furthermore, since
$SO_2$ is aspherical it follows that the composite of the
universal covering space projection $\widetilde{SF}_2\rightarrow
SF_2$ and the canonical map $SF_2\rightarrow SF_2/SO_2$ is a
homotopy equivalence.  Thus we have shown that the images of
$[X,SG_3]$ and $[X,SF_2]$ in $[X,G/O]$ are equal.  Finally, since
the image of $[X,SF_2]$ lies between the images of
$[X,\widetilde{SF}_2]$ and $[X,G_3]$, it follows that the images
of all three of these groups in $[X,G/O]$ must coincide.
\end{proof}

\begin{prop}
Let $p$ be an odd prime, let $k\geq 2$, and let $\alpha\in [X,SG]$ be
an element of order $p^r$ for some $r>0$.  Then $\alpha$ lies in
the image of $[X,SG_{2k}]\rightarrow [X,SG]$ if and only if
$\alpha$ corresponds to an element in the image of
$[S^{2k-1}X,S^{2k-1}]\rightarrow \{X,S^0\}$.
\end{prop}

\begin{proof}
We shall work with $p$-localization in the sense of Sullivan
\cite{Sul2}. Since connected $H$-spaces and simply connected spaces
all have good localizations at
$p$, it is meaningful to discuss the localized spaces
$$SG_{(p)},\quad SG_{2k(p)}, \quad SF_{2k-1(p)},
\quad S^{2k-1}_{(p)}, \quad SO_{2k(p)}, \quad
\Omega^{\infty}_0 S^{\infty}_{(p)} \quad \mathrm{and} \quad
\Omega^{2k-1}_0 S^{2k-1}_{(p)}$$
where $\Omega^m_0 Y$ denotes the path component
of the constant map in the iterated loop space $\Omega^m Y$.  Note
that if $W$ is an arcwise connected $H$-space whose homotopy groups
are all finite, then $W$ is naturally homotopy equivalent to the
weak product of its localizations $W_{(q)}$ at all primes $q$; in
particular, this applies to the $H$-spaces $SG\simeq
\Omega^{\infty}_0 S^{\infty}$ and $SF_{2k-1}\simeq \Omega^{2k-1}_0
S^{2k-1}$.

Recall that we have a fibration $SO_{2k-1}\rightarrow
SO_{2k}\rightarrow S^{2k-1}$ and that the tangent bundle
$T(S^{2k})$ is classified by a map $S^{2k-1}\rightarrow SO_{2k}$
such that the composite $S^{2k-1}\rightarrow SO_{2k}\rightarrow
S^{2k-1}$ has degree 2.  If we compose the map
$S^{2k-1}\rightarrow SO_{2k}$ with the inclusion of $SO_{2k}$ in
$SG_{2k}$ and the fibration $SG_{2k}\rightarrow S^{2k-1}$, the
resulting composite also have degree 2.  Therefore the map
\[
\begin{CD}
SF_{2k-1}\times S^{2k-1} @>>> SG_{2k}\times SG_{2k} @>
\mathrm{mult.} >> SG_{2k}
\end{CD}
\]
becomes a homotopy equivalence when localized at the odd prime
$p$.  Since the composite
$$S^{2k-1}\rightarrow SO_{2k}\rightarrow SO$$
is nullhomotopic, it follows that the image of $[X,SG_{2k(p)}]$
in $[X,SG_{(p)}]\cong [X,SG]_{(p)}$ is equal to the image of
$[X,SF_{2k-1(p)}]$, which corresponds to the image of
$[S^{2k-1}X,S^{2k-1}_{(p)}]\cong [S^{2k-1}X,S^{2k-1}]_{(p)}$ in
$\{X,S^0\}_{(p)}$; note that the codomain is the Sylow
$p$-subgroup of $\{X,S^0\}$ with respect to the loop sum, and
likewise the domain is the Sylow $p$-subgroup of the finite group
$[S^{2k-1}X,S^{2k-1}]$.  These observations imply that if
$\alpha\in[X,SG]$ is $p$-primary with respect to the composition
product (which is homotopy abelian) and lifts to $[X,SG_{2k}]$,
then $\alpha$ corresponds to an element of $\{X,S^0\}$ which
desuspends to $[S^{2k-1}X,S^{2k-1}]$.
\end{proof}

We shall also need the following result:

\begin{prop}\label{prop35}
If $\alpha\in [X,SG]$ has odd order and lies in the image of
$[X,SG_3]$, then the image of $\alpha$ in $[X,G/O]$ is trivial.
\end{prop}

\begin{proof}
Since the finite abelian group $[X,SG]$ splits into a product of
the groups $[X,SG]_{(q)}$, where $q$ runs through all primes,
it will suffice to prove the result when the order of $\alpha$ is
a power of some {\sf odd} prime $p$.

By Proposition~\ref{prop33} we have $\alpha = \alpha_1+\alpha_2$ where
$\alpha_2$ lies in the image of $[X,SO]\rightarrow [X,SG]$
and $\alpha_1$ corresponds to an element in the image of
$[S^2X,S^2]\rightarrow \{X,S^0\}$ or equivalently an element in
the image of $[X,SF_2]\to[X,G/O]$.  Since $\alpha_2$ maps trivially
into $[X,G/O]$, it suffices to prove the result for $\alpha_1$, so
we might as well assume that $\alpha$ itself lies in the image of
$[X,SF_2]$.

Standard homotopy computations imply that $\pi_1(SF_2)\cong\Z$ and
the higher homotopy groups of $SF_2$ are finite (look at the homotopy
groups of $S^2$).  Furthermore, the inclusion $S^1=SO_2\subset
SG_2\subset SF_2$ induces an isomorphism of
fundamental groups, and this yields a homotopy equivalence
$SF_2\simeq SO_2\times \widetilde{SF_2}$, where the second factor
denotes the universal covering space of $SF_2$.  Since the composite
$$SO_2\longrightarrow SO_3\longrightarrow SG_3 \longrightarrow SG_3/SO_3
\longrightarrow G/O$$
is trivial, it is enough to prove the result when $\alpha$ lies
in the image of $[X,\widetilde{SF}_2]$, and since the homotopy
groups of $\widetilde{SF}_2$ and $\widetilde{\Omega^2_0S^2}$ are
finite, we can say that $\alpha$ corresponds to a class in
$\{X,S^0\}_{(p)}$ which lies in the image of $[S^2X,S^2]_{(p)}$.

If $h: S^3\rightarrow S^2$ is the Hopf bundle whose fiber is $S^1$,
then composition with $h$ defines a homotopy equivalence from
$\Omega^2_0 S^3$ to $\Omega^2_0 S^2$, where as before $\Omega^m_0
Y$ denotes the path component of the constant map in $\Omega^m Y$.
Therefore, it follows that $\alpha$ corresponds to a class in
$\{X,S^0\}_{(p)}$ which lies the in image of the composite
$$
h: [S^2X,S^3]_{(p)}\rightarrow [S^2X,S^2]_{(p)}
$$
and hence $\alpha$ factors homotopically as a composite
$\overline{h}\tinycirc\beta$, where $\beta$ lies in $\{X,S^1\}_{(p)}$
and $\overline{h}$ denotes the image of $h$ in the stable group
$\{S^3,S^2\}_{(p)}\cong \{S^1,S^0\}_{(p)}$.  Finally, since
$\{S^1,S^0\}\cong\mathbb{Z}_2$ we have $\{S^1,S^0\}_{(p)} = 0$,
and hence $\alpha$ corresponds to the trivial element of
$\{X,S^0\}_{(p)}$, where we interpret the latter as a subgroup of
$\{X,S^0\}$.
\end{proof}

\begin{cor}
The conclusion of Proposition~{\rm \ref{prop35}} also
holds if we replace $G/O$ by $G/\mathrm{PL}$
or $G/\mathrm{Top}$.
\end{cor}

\section{Normal Invariants for Tangential Homotopy Lens Spaces}

Throughout this section $p$ will denote a fixed odd prime.

If $f: M\rightarrow N$ is a homotopy equivalence of compact
topological manifolds (possibly with boundary) and
$\eta(f)\in[N,G/\mathrm{Top}]$ is its normal invariant, then $f$
is a tangential homotopy equivalence if and only if a canonical
map from $[N,G/\mathrm{Top}]$ to $[N,B\,\mathrm{STop}]$ sends
$\eta(f)$ to zero ({\it e.g.\/}, see \cite{MadMilJ}), and by
the exactness of the fibration sequence $SG\rightarrow
G/\mathrm{Top}\rightarrow B\,\mathrm{STop}$ the image vanishes if
and only if $\eta(f)$ lies in the image of the associated map from
$[N,SG]$ to $[N,G/\mathrm{Top}]$.  In this section we shall
describe this image when $N$ is a $\Z_p$ lens space.  Our analysis
is based upon fundamental results on the structure of the
localized spaces $SG_{(p)}$, $G/O_{(p)}$, $G/\mathrm{Top}_{(p)}$,
$BSO_{(p)}$, $B\,\mathrm{STop}_{(p)}$ and similar objects; some
basic references are \cite{MadMilJ}, Chapter V of \cite{May} and
Lecture 4 of \cite{Ad4}.

We are particularly interested in the structure of $SG_{(p)}$ and
$G/\mathrm{Top}_{(p)}$.  Results of Sullivan (compare
\cite{MadMilJ}) imply that the localized spaces $BSO_{(p)}$ and
$G/\mathrm{Top}_{(p)}$ are homotopy equivalent and that
$G/O_{(p)}$ is homotopy equivalent to $BSO_{(p)}\times \CokJ_{(p)}$
for some space $\CokJ_{(p)}$ (see \cite{MadMilJ} for the definition
of the latter). Furthermore, if $J_p$ is defined as the fiber of the
map $\psi^r-1: BSO_{(p)}\rightarrow BSO_{(p)}$, where $r$ is a
primitive root of unity mod $p^2$ and $\psi^r$ is the Adams
operation in $K$-theory, then there is a homotopy equivalence from
$SG_{(p)}$ to $J_p\times \CokJ_{(p)}$ such that the following
diagram is homotopy commutative:

\[
\begin{CD}
SG_{(p)} @>>> G/O_{(p)} @>>> G/\mathrm{Top}_{(p)}\\
@V \simeq VV @V \simeq VV @V \simeq VV\\
J_p\times \CokJ_{(p)} @>\beta\times 1>> BSO_{(p)}\times \CokJ_{(p)}
@>\varphi>> BSO_{(p)}
\end{CD}
\]

In this diagram $\beta: J_p\rightarrow BSO_{(p)}$ is the homotopy
fiber of $\psi^r-1$ and $\varphi$ factors up to homotopy as a
composite
\[
\begin{CD}
BSO_{(p)}\times \CokJ_{(p)} @> \mathrm{proj.} >> BSO_{(p)} @>
\varphi '
>> BSO_{(p)}
\end{CD}
\]
because ${\widetilde{K}}(\CokJ_p) = 0$ ({\it cf. } \cite{MadMilJ},
Theorem 5.22, p. 115).

Since $[N,J_p]$ can often be computed fairly directly but
$[N,\CokJ_{(p)}]$ generally cannot, these splittings are very
helpful for describing the image of $[N,SG_{(p)}]$ in
$[N,G/\mathrm{Top}_{(p)}]$.

In addition to the splittings described above, there are also
splittings of $BSO_{(p)}$ that will be useful in this section.  The
results of \cite{MadMilJ} and \cite{Ad4} imply that the
localized complex $K$-theory spectrum $K_{(p)}$ splits into a sum
of $(p-1)$ periodic spectra $E_{\alpha} K_{(p)}$, where $\alpha$
runs through the elements of $\Z_{p-1}$.  Each of these spectra is
periodic of period $2p-2$, and the coefficient groups $E_\alpha
K_{(p)}(S^n)$ are given by $\Z_{(p)}$ if $n\equiv 2\alpha
\mathrm{~mod~}\ 2p-2$ and zero otherwise.  If we view the localized
real $K$-theory spectrum $KO_{(p)}$ as the direct summand given by
the self-conjugate part of $K_{(p)}$, then $KO_{(p)}$ corresponds
to the sum of the spectra $E_\beta K_{(p)}$, where $\beta$ runs
through all the {\sl even\/} elements of $\Z_{p-1}$.  We shall be
particularly interested in $E_0 K_{(p)}$, which is also known as the first
($p$-local) Morava $K$-theory $K(1)$, with $K(1)(S^n)=\Z_{(p)}$ if
$n\equiv 0\ \ \mathrm{mod}\ 2p-2$ and 0 otherwise (see \cite{Wur} for
background on Morava $K$-theories).

If $L$ is a $\Z_p$ lens space, then there is a canonical map $k_L$
from $L$ to the classifying space $B\Z_p$, and our analysis of the
image of $[L,SG]\rightarrow [L,G/\mathrm{Top}]$ begins with a
study of the analogous problem with $B\Z_p$ replacing $L$.  Both
$[B\Z_p,SG]\cong \{B\Z_p,S^0\}$ and $[B\Z_p,BSO]\cong
\widetilde{KO}(B\Z_p)$ are well understood; results of D.W.
Anderson (summarized in \cite{And1}, with full details in \cite{And2})
imply that the latter (with the group operation given by direct sum)
is algebraically
isomorphic to the completion of the ideal $IO(\Z_p)$ in the real
representation ring $RO(\Z_p)$ given by all 0-dimensional virtual
representations (this also follows directly from \cite{At}), while
the proof of the Segal Conjecture for $\Z_p$ (see \cite{Rav2} and
\cite{AdGunMill}) implies that $\{B\Z_p,S^0\}$ is algebraically
isomorphic to
the completion of the ideal $IA(\Z_p)$ in the Burnside ring
$A(\Z_p)$ given by all virtual finite $\Z_p$-sets with virtual
cardinality 0. Although the set theoretic isomorphism from
$\{B\Z_p,S^0\}$ to $[B\Z_p,SG]$ is not additive, one can prove
that the latter is also algebraically isomorphic to the completion
of $IA(\Z_p)$
using the methods of \cite{At}, and this is explained in
\cite{Lai}. The ideal $IA(\Z_p)$ is infinite cyclic and it turns
out that the image of one generator in the completed ideal
$\widehat{IA(\Z_p)}\cong \{B\Z_p,S^0\}$ corresponds to the reduced
stable homotopy-theoretic transfer $B\Z_p\rightarrow S^0$
associated to the standard $p$-fold covering $E\Z_p\rightarrow
B\Z_p$ (see \cite{KP1} and \cite{KP2}), whose total space is contractible.

It is fairly straightforward to prove that the completion
$\widehat{IA(\Z_p)}$ is topologically and additively isomorphic to
the additive $p$-adic integers $\widehat{\Z_{(p)}}$ and
$\widehat{IO(\Z_p)}$ is similarly isomorphic to a sum of
$\frac{1}{2}(p-1)$ copies of the $\widehat{\Z_{(p)}}$.  One can
describe these groups and their interrelationships more precisely
as follows:

Let $I(\Z_p)$ be the ideal in the complex representation ring
$R(\Z_p)$ given by the kernel of the virtual dimension map from
$R(\Z_p)$ to $\Z$.  Then $\widetilde{K}(B\Z_p)$ is isomorphic to
the completion $\widehat{I(\Z_p)}$ by \cite{At}, and hence it is a
free $\widehat{Z_p}$-module on $(p-1)$ generators.  We can choose
these free generators to have the form $\mathbf{e}_a$, where $a$
runs through the nonzero elements of $\Z_p$, and if $r$ is a
primitive root of unity mod $p^2$ then the Adams operation
$\psi^r$ on $\widetilde{K}(B\Z_p)$ sends $\mathbf{e}_a$ to
$\mathbf{e}_{ra}$; furthermore, if $\theta$ is the additive
automorphism of $\Z_p$ sending $a\in\Z_p$ to $ra$ and $B\theta$ is
the induced self-map of $B\Z_p$ which induces $\theta$ on the
fundamental group level (so $B\theta$ is unique up to homotopy)
then the induced automorphism $B\theta^*$ in $K$-theory also sends
$e_a$ to $e_{ra}$.  Furthermore, the complexification map from
$\widetilde{KO}(B\Z_p)$ to $\widetilde{K}(B\Z_p)$ is split
injective, and its image is the free submodule whose generators
have the form $e_a+e_{-a}$, where $a$ runs through all nonzero
elements of $\Z_p$ (note that there are $\frac{1}{2}(p-1)$
elements of this form).  With this background, we can describe a
canonical homomorphism from $[B\Z_p,SG]$ to $[B\Z_p, BSO]$ as
follows:

\begin{prop}\label{prop41}
Let $F: SG_{(p)}\rightarrow BSO_{(p)}$ be the composite
\[
\begin{CD}
SG_{(p)} @>>> G/O_{(p)}\simeq BSO_{(p)}\times \CokJ_p @>>>
BSO_{(p)}
\end{CD}
\]
where the final arrow is coordinate projection.  Then the image of
$[B\Z_p,SG]\cong [B\Z_p,SG_{(p)}]$ in $[B\Z_p,BSO]\cong
[B\Z_p,BSO_{(p)}]$ corresponds to the split free submodule of
$\widetilde{K}(B\Z_p)\cong \bigoplus^{p-1}\widehat{\Z_{(p)}}$
generated by the sum of the basis elements $\sum_{a}
\mathbf{e}_a$, and the image corresponding to the direct summand
$E_0 K_{(p)}$ in $KO(B\Z_p)\cong KO_{(p)}(B\Z_p)$.
\end{prop}

\begin{proof}
If $V: G/O\rightarrow BSO$ is the homotopy fiber of
$BSO\rightarrow BSG$, then by construction the composite
\[
\begin{CD}
BSO_{(p)} @>\mathrm{slice}>> BSO_{(p)}\times \CokJ_p\simeq
G/O_{(p)} @> V_{(p)} >> BSO_{(p)}
\end{CD}
\]
\noindent
is given by $\psi^r - 1$.  Now $V_{(p)}$ is trivial on $\CokJ_p$,
and therefore for all connected CW complexes the image of $[X,SG_{(p)}]$ in
$[X,G/O_{(p)}]$ will be the kernel of the map $V_{(p)*}:
[X,G/O_{(p)}]\rightarrow [X,BSO_{(p)}]$.  If we combine these we
see that the kernel of $V_{(p)*}$ is generated by $[X,\CokJ_p]$
and the kernel of $\psi^r - 1$ on $\widetilde{KO}_{(p)}(X)$.  If
we let $X=B\Z_p$, then the localized and unlocalized groups are
isomorphic, and if we expand an element $\xi$ of
$\widetilde{KO}_{(p)}(B\Z_p)$ as $\sum c_a \mathbf{e}_a$ for
suitable coefficients $c_a$ (note that $c_a = c_{-a}$), then
$\xi$ lies in the kernel of $\psi^r - 1$ if and only if $c_{ra} =
c_a$ for all $a$.  We claim this happens if and only if the
coefficients $c_a$ are all equal.  Sufficiency is obvious; on the
other hand, it follows by induction that $c_{r^ka} = c_a$ for all
$k$ and $a$, and since the powers $r^k$ exhaust the nonzero
elements of $\Z_p$ we must have $c_a = c_b$ for all $a,b\neq 0$.
If we now denote the image of $[B\Z_p,SG]$ in
$\widetilde{KO}(B\Z_p)$ as $M$, the preceding discussion shows
that $M$ is a direct summand of $\widetilde{KO}(B\Z_p)$ which is
isomorphic to $\widehat{\Z_{(p)}}$ and the complementary summand
$M'$ is a free $\widehat{\Z_{(p)}}$-module on $(p-2)$ generators.
In particular, $M'\cong \widetilde{KO}(B\Z_p)/M$ is torsion free.

\begin{clm}
$M$ is contained in the summand $E_0 K_{(p)}(B\Z_p)$
\end{clm}

\begin{proof}
To see this, let $E_0^\bot K_{(p)}$ denote the sum of the other
cohomology theories $E_i K_{(p)}$, and let $\overline{M}$ denote
the projection of $M$ onto $E_0^\bot K_{(p)}$ with respect to the
splitting $\widetilde{KO}(B\Z_p)\cong E_0 K_{(p)}(B\Z_p)\oplus
E_0^\bot K_{(p)}(B\Z_p)$.  We know that $\psi^r - 1$ restricted to
$M$ is trivial, but we also know that $\psi^r - 1$ restricted to
$E_0^\bot K_{(p)}(B\Z_p)$ is injective (compare \cite{MadMilJ}), and
these combine to imply that $\overline{M}$ is trivial, so that $M$
must be contained in $E_0 K_{(p)}(B\Z_p)$.
\end{proof}

The results of \cite{Kam} imply that the summand $E_0
K_{(p)}(B\Z_p)$ of $\widetilde{K}(B\Z_p)$ must also be isomorphic
to $\widehat{\Z_{(p)}}$, and the complementary summand $E_0^\bot
K_{(p)}(B\Z_p)$ must be torsion free.  Therefore the quotient
$\widetilde{KO}(B\Z_p)/M$ is isomorphic to the direct sum of
$E_0^\bot K_{(p)}(B\Z_p)$ and a quotient $M_1\cong
\widehat{\Z_{(p)}}/M$, where $M$ is also isomorphic to
$\widehat{\Z_{(p)}}$.  If $M$ is a proper subgroup of
$\widehat{\Z_{(p)}}$, then the quotient $M_1$ must be a nontrivial
finite cyclic $p$-group, and therefore the quotient $M_1$ has
nontrivial elements of finite order.  Since $M$ is a direct summand
of $\widetilde{KO}(B\widehat{\Z_{(p)}})$ and the latter is a
direct sum of $\frac{1}{2}(p-1)$ copies of $\widehat{\Z_{(p)}}$,
this cannot happen and hence we must have $M=E_0 K_{(p)}(B\Z_p)$.
\end{proof}

We also have a similar conclusion regarding the image of
$[B\Z_p,SG]$ in $[B\Z_p,G/\mathrm{Top}]\cong
[B\Z_p,G/\mathrm{Top}_{(p)}]$; as noted before, the results of
Sullivan show the codomain is isomorphic to $\widetilde{KO}(B\Z_p)$.

\begin{prop}\label{prop42}
The image of the composite
\[
\begin{CD}
[B\Z_p,SG] @>>> [B\Z_p,G/\mathrm{Top}]~~\cong~~ \widetilde{KO}(B\Z_p)
\end{CD}
\]
is the image of $E_0 K_{(p)}(B\Z_p)$, and the kernel of this map
is trivial.
\end{prop}

\begin{proof}
To prove this first statement, we need to analyze the image of the
map
\[
E_0 K_{(p)}(B\Z_p)~~\to~~\widetilde{KO}(B\Z_p)~~\subseteq~~
[B\Z_p,G/O_{(p)}]~~\to~~[B\Z_p,G/\mathrm{Top}_{(p)}]~~\cong~~
\widetilde{KO}(B\Z_p)
\]
and the composite of this map with the splitting retraction
$\widetilde{KO}(B\Z_p)\rightarrow E_0 K_{(p)}(B\Z_p)$.  We claim
this composite is an isomorphism. Since the composite is given by
a natural transformation of cohomology theories, it suffices to
show that this transformation induces an isomorphism of cohomology
theories, and the latter in turn reduces to showing that the
induced self-maps of the localized homotopy groups
$\pi_{2k(p-1)}(BSO)_{(p)}$ are isomorphisms.

We have the following commutative diagram, in which each group
except $\pi_{2k(p-1)}(BJ_p)$ is a direct sum of $\Z_{(p)}$ and a
finite abelian $p$-group:

\[
\begin{CD}
\pi_{2k(p-1)}(BSO)_{(p)} &
@>\mathrm{splitting}>\mathrm{injection}> \pi_{2k(p-1)}(G/O)_{(p)}
&
@>\beta^O_*>> \pi_{2k(p-1)}(BSO)_{(p)} & @>e_*>> \pi_{2k(p-1)}(J_p)\\
@V \theta VV & @V \theta' VV & @VVV & @VV = V\\
\pi_{2k(p-1)}(BSO)_{(p)} & @>\cong>\mathrm{Sullivan}>
\pi_{2k(p-1)}(G/\mathrm{Top})_{(p)} & @>\beta^{\mathrm{Top}}_*>>
\pi_{2k(p-1)}(B\,\mathrm{STop})_{(p)} & @>>> \pi_{2k(p-1)}(J_p)
\end{CD}
\]

As noted in \cite{MadMilJ} (see p. 117), modulo torsion the two
vertical arrows on the left are multiplication by the number
explicitly given on that page.  The
maps $\beta^O_*$ and $\beta^{\mathrm{Top}}_*$ give the underlying
bundles, and the first line is exact $\beta^O_*$ and $e_*$;
furthermore, up to units in $\Z_{(p)}$, the map $\beta^O_*$ is
multiplication by the order of the inverse of the
$J$-homomorphism, and this order is divisible by $p$.  On the
other hand, up to torsion and units in $\Z_{(p)}$, the maps
$\pi_{2k(p-1)}(G/O)_{(p)}\rightarrow
\pi_{2k(p-1)}(G/\mathrm{Top})_{(p)}$ and $\pi_{2k(p-1)}(BSO)_{(p)}
\rightarrow \pi_{2k(p-1)}(B\,\mathrm{STop})_{(p)}$ are
multiplication by
$$c=(2^{k(p-1)-1}-1)\cdot \textsc{num}\left(
{\frac {B_{k(p-1)}}{2kp-2k}}\right)$$
where ``$\textsc{num}(...)$'' denotes the numerator of a
fraction reduced to least terms and $B_{k(p-1)}$ is the appropriate
Bernoulli number
(see \cite{MadMilJ}, p. 117,  for the first map and \cite{Br} for
the second). Therefore, modulo torsion, the right hand square is
given by
\[
\begin{CD}
\Z_{(p)} @>>> \Z_{p^m}\\
@V c VV @VV = V\\
\Z_{(p)} @>>> \Z_{p^m}
\end{CD}
\]
\noindent
where the top and bottom arrows are epimorphisms and the top arrow
is the standard quotient projection.  Such a diagram can exist only
if $c$ is relatively prime to $p$, and therefore the two vertical
arrows $\theta$ and $\theta'$ in the large previous diagram are
isomorphisms; this is what we wanted to prove.

Finally, we need to check that the image of $[B\Z_p,SG]$ in
$[B\Z_p,G/\mathrm{Top}]$ corresponds to $E_0 K_{(p)}(B\Z_p)$ and
that $[B\Z_p,SG]$ is mapped isomorphically onto its image.  By the
preceding discussion we know that this image is a direct summand
of $[B\Z_p,G/\mathrm{Top}]\cong\widetilde{KO}(B\Z_p)$ and is
isomorphic to $\widehat{\Z_{(p)}}$.  Thus the map from
$[B\Z_p,SG]$ to its image is given by a surjective homomorphism
from $\widehat{\Z_{(p)}}$ to itself.  Since every such surjection
is an isomorphism, we see that $[B\Z_p,SG]$ must be mapped
isomorphically to its image.  To prove that the image in
$[B\Z_p,G/\mathrm{Top}]\cong \widetilde{KO}(B\Z_p)$ is $E_0
K_{(p)}(B\Z_p)$ we can use the reasoning in the proof of
Proposition 4.1 to reduce the question to checking that the image
of $[B\Z_p,SG]$ in $[B\Z_p,G/\mathrm{Top}]\cong
\widetilde{KO}(B\Z_p)$ is contained in the kernel of $\psi^r - 1$.
We have already noted that on $\widetilde{KO}(B\Z_p)$ one has
$\psi^r = B\theta^*$ for some automorphism $\theta$ of $\Z_p$, so
everything reduces to showing that the map
\[
\begin{CD}
\widetilde{KO}(B\Z_p)~\subseteq~ [B\Z_p,G/\mathrm{Top}] @>>>
[B\Z_p,G/\mathrm{Top}]~\cong~ \widetilde{KO}(B\Z_p)
\end{CD}
\]
sends the kernel of $B\theta^* - 1$ to itself.  Since the
displayed map arises from some self-map of $BSO_{(p)}$, it follows
immediately that this mapping does send the kernel to itself,
proving the remaining assertions in the proposition.
\end{proof}

Now let $L$ be a $(2n-1)$-dimensional lens space, and let $\eta_L:
L\rightarrow B\Z_p$ be its classifying map.  We may assume that
$B\Z_p$ is constructed so that its $(2n-1)$-skeleton is $L$, and
we shall do so henceforth.  Our next objective it to derive
analogs of Propositions 4.1 and 4.2 in which $B\Z_p$ is replaced
by $L$.

More precisely, we need to extend our observations about the map
$$[B\Z_p,SG]\rightarrow [B\Z_p,G/\mathrm{Top}]$$
into an effective
analysis of all the objects and morphisms in the following commutative
diagram:
\[
\begin{CD}
[B\Z_p,SG] @>>> [B\Z_p,G/\mathrm{Top}] @>\cong>>
\widetilde{KO}(B\Z_p)\\
@V\eta^*VV @V\eta^*VV @V\eta^*VV\\
[L,SG] @>>> [L,G/\mathrm{Top}] @>\cong>> \widetilde{KO}(L)
\end{CD}
\]
The results of \cite{Kam} show that the induced map in reduced
$KO$-theory is surjective and yield an explicit description of its
kernel; the following result on the summand $E_0K_{(p)}$ is a
straightforward consequence of the methods and conclusion of
\cite{Ad4}:

\begin{prop}
If $L$ is a $(2n-1)$-dimensional $\Z_p$ lens space, then the
Atiyah-Hirzebruch spectral sequence of $E_0K_{(p)}(L)$ collapses,
and this group is cyclic of order $p^m$ where
$m=\left[\frac{n-1}{p-1}\right]$ (and $[ \cdots ]$ denotes the greatest
integer function).
\end{prop}

\begin{proof}{$(Sketch)$}\quad
The spectral sequence collapses because the analogous spectral
sequence for $\widetilde{KO}_{(p)}(L)$ collapses (see \cite{At} or
\cite{Kam}), and the cyclic nature of the group follows because
$\widetilde{KO}(B\Z_p)\rightarrow KO_{(p)}(L)$ is onto (see
\cite{Kam}) and $E_0K_{(p)}$ is a direct summand of $KO$.
\end{proof}

In contrast to this result, the map from $[B\Z_p,SG]$ to $[L,
SG_{(p)}]$ is not necessarily onto, but we shall show that the
image is a natural direct summand which maps
onto $E_0K_{(p)}(L) \subseteq KO_{(p)}(L)\cong
[L,G/\mathrm{Top}_{(p)}]$
with an easily described kernel,
and the complementary summand of
$[L,SG_{(p)}]$ maps to zero in $[L,G/\mathrm{Top}_{(p)}]$.  The
summands of $SG_{(p)}$ are given by the splitting $SG_{(p)}\simeq
J_p\times \CokJ_p$; results of \cite{May} imply this decomposition
comes from a splitting of infinite loop spaces.

Most of what we need to know about $[B\Z_p,J_p]\rightarrow
[L,J_p]$ is contained in the following results.

\begin{prop}
We have $[B\Z_p,\CokJ_p]=0$, and $[B\Z_p,SG]\cong [B\Z_p,SG_{(p)}]$ is
isomorphic to $[B\Z_p,J_p]$
\end{prop}

\begin{proof}{$(Sketch)$}\quad
Propositions 4.1 and 4.2 imply that the composite
\[
\begin{CD}
[B\Z_p,J_p]\times[B\Z_p,\CokJ_p]  @>>>
[B\Z_p,SG_{(p)}]\times[B\Z_p,SG_{(p)}]\\
@|  @VVV  @.\\
[B\Z_p,SG_{(p)}] @.
[B\Z_p,G/\mathrm{Top}_{(p)}] \times [B\Z_p,G/\mathrm{Top}_{(p)}] @>\oplus >>
[B\Z_p,G/\mathrm{Top}_{(p)}]
\end{CD}
\]

\noindent
in which the first horizontal arrow is inclusion, is split
injective, and since the composite $\CokJ_p\rightarrow SG_{(p)}
\rightarrow G/\mathrm{Top}_{(p)}$ is nullhomotopic the displayed
composite can be rewritten more simply as
\[
\begin{CD}
[B\Z_p,SG_{(p)}] @>\mathrm{proj}>> [B\Z_p,J_p] @>>>
[B\Z_p,G/\mathrm{Top}_{(p)}]
\end{CD}
\]
It follows that the projection map induces a split injection from
$[B\Z_p,SG_{(p)}]$ to $[B\Z_p,J_p]$.  Since the projection is onto
by construction, it follows that the map
$[B\Z_p,SG_{(p)}]\rightarrow [B\Z_p, J_p]$ is an isomorphism,
proving one assertion in the proposition.  To see that $[B\Z_p,\CokJ_p]=0$,
notice that if this group were nonzero then the projection
$[B\Z_p,SG_{(p)}]\rightarrow [B\Z_p,J_p]$ would not be injective.
\end{proof}

We are now ready to analyze objects like $[L,SG_{(p)}]$ and its
summands where $L$ is a lens space as above.

\begin{prop}
Let $L$ be a $(2n-1)$-dimensional $\Z_p$ lens space, let $\eta
:L\rightarrow B\Z_p$ denote its classifying map and let $q:
L\rightarrow S^{2n-1}$ be a map of degree $1$ (which is unique up to
homotopy).  Then $[L,J_p]$ is the sum of the image of $\eta^*:
[B\Z_p,J_p]\rightarrow [L,J_p]$ and $q^*:
\pi_{2n-1}(J_p)\rightarrow [L,J_p]$.  The image of $\eta^*$ is
cyclic of order $p^m$, where $m=\left[\frac{n}{p-1}\right]$, the map
$q^*$ is injective and the structures of $[L,J_p]$ and the map
$[L,J_p]\rightarrow [L,G/\mathrm{Top}]\cong \widetilde{KO}(L)$ are
given as follows:
\begin{itemize}
\item[(i)] Suppose that $n\not\equiv 0$\ \ mod $p-1$, so that
$\pi_{2n-1}(J_p)=0$.  Then $[B\Z_p,J_p]\rightarrow [L,J_p]$ is
onto and $[L,J_p]\rightarrow [L,G/\mathrm{Top}_{(p)}]\cong
\widetilde{KO}(L)$ is split injective with image corresponding to
$E_0K_{(p)}(L)$.  Furthermore, the latter also equals the image of
$[L,SG_{(p)}]\rightarrow [L,G/\mathrm{Top}_{(p)}]$, and this map
is a split injection.

\item[(ii)] Suppose that $n=p^\nu(p-1)r$ where $r$ is prime to
$p$, so that $\pi_{2n-1}(J_p)\cong \Z_{p^\nu}$.  Then the images
of $\eta^*$ and $q^*$ intersect in a subgroup of order $p$, the
image of $q^*$ is the kernel of the map $[L,SG_{(p)}]\rightarrow
[L,G/\mathrm{Top}_{(p)}]$ and the image of the latter is given by
$E_0K_{(p)}(L)$.  Furthermore, the latter is also equal to the
images ot $[L,J_p]\rightarrow [L,G/\mathrm{Top}_{(p)}]$ and the
kernel of this map has order $p$.
\end{itemize}
\end{prop}

\textbf{Note:} Since the group $[L,\CokJ_p]$ is usually
nontrivial, the map $[B\Z_p,SG_{(p)}]\rightarrow [L,SG_{(p)}]$ is
usually not onto; furthermore, since the homotopy groups of $\CokJ_p$
are given by largely unknown factors in the stable homotopy
groups of spheres, the groups $[L,\CokJ_p]$ are usually not easy
to describe explicitly.  This leads to major complications in
studying the notion of smooth tangential thickness for lens
spaces.

\textbf{Notation}: Given a CW complex $X$ and an arcwise connected
space $Y$, the {\it skeletal filtration} of the set of homotopy classes
$[X,Y]$ is the family of subsets
$$
\mathbf{XF}_k([X,Y])= \{u\in [X,Y]\ :\ u\,|\,{X_k}\hbox{ is
homotopically trivial}\}~.
$$
Clearly if $f: X'\rightarrow X$ is a cellular map, then the map
$f^*: [X,Y]\rightarrow [X',Y]$ is also filtration preserving.
Similarly, if $g: Y\rightarrow Y'$ is continuous, then $g_*:
[X,Y]\rightarrow [X,Y']$ is filtration preserving.  If $Y$ is a
double loop space, set
$$
\mathbf{FF}_k([X,Y]) = \mathbf{XF}_k([X,Y])\smallsetminus
\mathbf{XF}_{k-1}([X,Y])
$$
and note that $(i)$ this has a natural abelian group structure and
$(ii)$ $\mathbf{FF}_k([X,\Omega^2W])$ is functorial in the second
variable $W$.

\begin{proof}
The first step is an anlysis of the Atiyah-Hirzebruch spectral
sequence for $[B\Z_p,J_p]$ and $[L,J_p]$.

\textbf{Claim:} \textit{The Atiyah-Hirzebruch spectral sequence
for $[B\Z_p,J_p]$ collapses.}

\begin{proof}
The relevant $E_2$ terms are given by
$\widetilde{H}^i(B\Z_p;\pi_i(J_p))$; these groups are isomorphic
to $\Z_p$ if $i=2k(p-1)-1$ for some integer $k$ and zero
otherwise.  We also know that $[B\Z_p,J_p]$ maps isomorphically to
$E_0K_{(p)}(B\Z_p)$, and the collapsing Atiyah-Hirzebruch spectral
sequence for the latter has $E_2$ terms given by
$\widetilde{H}^i(B\Z_p;E_0K_{(p)}(S^i))$, which are isomorphic to
$\Z_p$ if $i\equiv 0$  mod $2(p-1)$ and zero otherwise.  It is a
fairly straightforward exercise to check that the bijectivity of
$[B\Z_p,J_p]\rightarrow E_0K_{(p)}(B\Z_p)$ implies that the
spectral sequence for $[B\Z_p,J_p]$ must also collapse.
\end{proof}

Next, we shall use the naturality properties of the
Atiyah-Hirzebruch spectral sequence to analyze $[L,J_p]$ and
related objects.  Let $\lambda: S^{2n-1}\rightarrow L$ be the
universal covering projection; then the mapping cone $\widehat{L}$
of $\lambda$ can be viewed as the $2n$-skeleton of $B\Z_p$ and the
restriction $H^*(B\Z_p)\rightarrow H^*(\widehat{L})$ is an
isomorphism in dimensions $\leq 2n$ for all coefficients.
Therefore a naturality argument implies that the Atiyah-Hirzebruch
spectral sequence for $[\widehat{L},J_p]$ collapses and the
restriction $\widehat{\Z}_{(p)}\cong [B\Z_p,J_p]\rightarrow
[\widehat{L},J_p]$ is onto with image $\Z_{p^m}$ where
$m=\left[\frac{n}{p-1}\right]$.  Since $\pi_{2n}(J_p)=0$, the
Barratt-Puppe exact sequence associated to
\[
\begin{CD}
S^{2n-1} @>\lambda>> L @>>> \widehat{L} @>>> S^{2n} @>>> \cdots
\end{CD}
\]
implies that the restriction map $[\widehat{L},J_p]\rightarrow
[L,J_p]$ is injective, and hence the image of $[B\Z_p,J_p]$ in
$[L,J_p]$ is also cyclic of order $p^m$ (where $m$ is as given
above).  To describe the entire group $[L,J_p]$, let $L_0 =
L\smallsetminus\mathrm{Int}\ D$ where $D$ is a smoothly embedded closed
$(2n-1)$-disk; then $L_0$ may be viewed as a $(2n-2)$-skeleton for
$B\Z_p$ and $H^*(B\Z_p)\rightarrow H^*(L_0)$ is an isomorphism in
dimensions $\leq 2n-2$.  As before, it follows that
$[B\Z_p,J_p]\rightarrow [L_0,J_p]$ is onto.  If we consider the
exact sequence for the Barratt-Puppe sequence
\[
\begin{CD}
S^{2n-2}~~\subseteq~~ L_0~~ \subseteq L @>q>> S^{2n-1}
\end{CD}
\]
we see that if $y\in [L,J_p]$, then the restriction of $y$
to $[L_0,J_p]$ is the image of some class $z\in[B\Z_p,J_p]$ and
therefore $y-\eta^*z\in[L,J_p]$ must lie in the image of $q^*$.
Thus the images of $q^*$ and $\eta^*$ generate $[L,J_p]$.  To see
that $q^*$ is injective, we begin by noting that
$\pi_{2n-1}(J_p)=0$ unless $n\equiv 0$ mod $p-1$, and if
$n=p^\nu(p-1)r$, where $\nu\geq 1$ and $r$ is prime to $p$,
then $\pi_{2n-1}(J_p)$ is cyclic of order $p^\nu$ (see \cite{Ad1}).  If
$n\equiv 0$ mod $p-1$, then our computations of $[L,J_p]$ and
$[\widehat{L},J_p]$ show that $q^*$ maps a generator for the
$p$-torsion in $\pi_{2n-1}(J_p)\cong \Z_{p^\nu}$ to a class of
order $p$ in the image of $\eta^*: [B\Z_p,J_p]\rightarrow
[L,J_p]$.  In particular, $q^*$ maps the $p$-torsion injectively,
and hence it must map the entire cyclic $p$-group
$\pi_{2n-1}(J_p)$ injectively.  Observe that if $n\equiv 0$ mod
$p-1$, the preceding argument and skeletal filtration
considerations show that the intersection of Im $\eta^*$ and Im
$q^*$ is a cyclic subgroup of order $p$.

We must now describe the image of $[L,J_p]$ in
$[L,G/\mathrm{Top}_{(p)}]$.  First of all, we claim that the map
$[L,J_p]\rightarrow [L,G/\mathrm{Top}_{(p)}]$ is trivial on the
image of $q^*$.  More or less by construction, the composite maps
of homotopy groups $\pi_*(SO)\rightarrow \pi_*(SG)\rightarrow
\pi_*(J_p)$ are onto in all dimensions, and since the composite in
the commutative diagram
\[
\begin{CD}
SO @>>> SG_{(p)} @>>> G/O_{(p)}\\
@. @VVV @VVV\\
@. J_p @>>> G/\mathrm{Top}_{(p)}
\end{CD}
\]
is homotopically trivial, and therefore the composite
\[
\begin{CD}
\pi_{2n-1}(SO)_{(p)} @>\mathrm{onto}>> \pi_{2n-1}(J_p) @>>>
\pi_{2n-1}(G/\mathrm{Top}_{(p)})\\
@. @VV q^* V @VV q^* V\\
@. [L,J_p] @>>> [L,G/\mathrm{Top}_{(p)}]
\end{CD}
\]
must be zero.

The final step is to check that the morphism from the cyclic
$p$-group $[L,J_p]$ to the abelian $p$-group
$[L,G/\mathrm{Top}_{(p)}]$ maps onto the summand $E_0K_{(p)}(L)$
and the kernel is precisely the image of $q^*$.  Is is convenient
to split the discussion into two cases, depending on whether or
not $n\equiv 0$ mod $p-1$.  In both cases the argument uses the
commutative diagram
\[
\begin{CD}
[L,J_p] @>>> [L,G/\mathrm{Top}_{(p)}]\\
@VVV @VVV\\
[L_0,J_p] @>>> [L_0,G/\mathrm{Top}_{(p)}]
\end{CD}
\]
in which the vertical arrow on the right is an isomorphism by
standard results on $\widetilde{KO}(L)$ and $\widetilde{KO}(L_0)$
which follow from the collapsing of their respective Atiyah-Hirzebruch
spectral sequences.

\textit{Case} $(i)$:  If $\not\equiv 0$ mod $p-1$, then the
restriction map from $[L,J_p]$ to $[L_0,J_p]$ is an isomorphism,
and the restrictions $[B\Z_p,J_p]\rightarrow [L_0,J_p]$ and
$[B\Z_p,G/\mathrm{Top}_{(p)}]\rightarrow
[L_0,G/\mathrm{Top}_{(p)}]$ are onto with isomorphic images.  A
diagram chase now shows that $[L_0,J_p]\rightarrow
[L_0,G/\mathrm{Top}_{(p)}]$ is a split injection whose image is
$E_0K_{(p)}(L)$, which is what we wanted to prove.

\textit{Case} $(ii)$:  If $n\equiv 0$ mod $p-1$, then the kernel
of the restriction map from $[L,J_p]$ to $[L_0,J_p]$ is the image
of $q^*$, and the kernel of the map from Im $\beta^*$ to
$[L_0,J_p]$ has order $p$.  As in the preceding case, the map
$[L_0,J_p]\rightarrow [L_0,G/\mathrm{Top}_{(p)}]$ is an
isomorphism, so the conclusion in this case also follows from a
diagram chase.

To see the statements about the images of $[L,SG_{(p)}]$ and
$[L,J_p]$ in $[L,G/\mathrm{Top}_{(p)}]$, note that these images
are equal by the splitting $SG_{(p)}\simeq J_p\times \CokJ_p$ and
the homotopic triviality of $\CokJ_p\rightarrow
G/\mathrm{Top}_{(p)}$.
\end{proof}

\section{Weak Rigidity for Self-Equivalences of Lens Spaces}

We have defined tangential thickness for elements of the
surgery structure set of a manifold $M$, but sometimes it is
more desirable to have a similar notion of tangential thickness
for manifolds without the additional data coming from a
homotopy equivalence $f:N\to M$.  This is not always feasible,
but in some cases a simple condition on $M$ leads to
results which can be stated fairly simply.

\textbf{Definition. }
Let CAT be the category of smooth, piecewise linear or
topological manifolds.  A compact unbounded CAT-manifold $M$ is
said to be {\sl weakly tangentially rigid in } 
CAT if every tangential
homotopy self-equivalence $h:M\to M$ in CAT is normally cobordant to the
identity in the CAT-structure set of $M$; recall that a homotopy 
equivalence $f:N\to M$ of CAT manifolds is said to be tangential in CAT 
provided the pullback of the stable CAT tangent bundle of $M$ is 
CAT isomorphic to the stable CAT tangent bundle 
in $M$.  

\begin{prop}\label{prop:welldef}
Let $M$ be a weakly tangentially rigid closed manifold in the
category of topological manifolds,
let $f,f':N\to  M$ be tangential homotopy equivalences, and let
$k\geq 3$ be a positive integer.  Then $(N,f)$ determines a manifold
structure which lies in $\TT^{\mathrm{Top}}_k(M)$ if and only if
$(N,f')$ does.
\end{prop}

\begin{proof}
Suppose that $h:M\to M$ is a tangential homotopy self-equivalence,
so that $h$ is normally cobordant to the identity by the
weak tangentially rigidity assumption.
Since $k\geq 3$ the $\pi-\pi$ Theorem
implies that $h\times\mathrm{Id}(D^k):M\times D^k\to M\times D^k$
is properly homotopic (as a map of manifold pairs) to a
homeomorphism, and if we restrict everything to the interiors
of the relevant manifolds we see that
$f\times\mathrm{Id}(\R^k):M\times \R^k\to M\times \R^k$
is also properly homotopic to a homeomorphism which we shall
call $H_k$.

Suppose now that
$f,f':N\to  M$ are tangential homotopy equivalences.
By the symmetry of the conclusion it suffices to prove the result
when $(N,f)$ determines a manifold structure in
$\TT^{\mathrm{Top}}_k(M)$.

If $g$ is a homotopy inverse to $f'$ and $h=f'\tinycirc g$,
then $h$ is a tangential homotopy self-equivalence of $M$ and
$f'$ is homotopic to $h\tinycirc f$.  By the weak tangential rigidity
assumption, we know that $h\times\mathrm{Id}(\R^k)$ is homotopic to a
homeomorphism $H_k$ as described above. By the assumption on $(N,f)$
there is a homeomorphism
$F:N\times\R^k\to M\times\R^k$ which homotopically corresponds to
$f$ under the canonical homotopy equivalences $X\simeq X\times\R^k$.
It now follows that $H_k\tinycirc F$
is a homeomorphism
$N\times\R^k\to M\times\R^k$ which homotopically corresponds to
$f'$ under the canonical homotopy equivalences $X\simeq X\times\R^k$,
and this means that $(N,f')$ determines a class in
$\TT^{\mathrm{Top}}_k(M)$.
\end{proof}

The next result shows that everything in this section applies to
lens spaces.

\begin{prop}\label{prop:rigid}
Let $M^{2n-1}$ be a lens space with fundamental group $\Z_p$,
where $p$ is an odd prime, and let {\rm Top} be the topological 
manifold category. Then
every {\rm Top} tangential homotopy self-equivalence of $M^{2n-1}$ is
{\rm Top} normally cobordant to the identity.
\end{prop}

There are lens spaces which have tangential homotopy
self-equivalences that are not homotopic to the identity.  For
example, if $p\equiv 1$ mod 4 is a prime then the simple lens
space $L^3(p;1,1)$ admits a tangential homotopy self-equivalence
which is not homotopic to a homeomorphism and induces multiplication
by $v$ on the fundamental group, where $v^2\equiv~-1$ mod $p$ (this
follows from \cite{Rueff} and Reidemeister torsion considerations
as in Section 12 of \cite{Milnor}).

\begin{proof}
Let $f:M\to M$ be a tangential homotopy self-equivalence of $M$,
and let $\eta_T(f)\in [M,G/\mathrm{Top}]_{(p)}$ denote its normal
invariant; since $f$ is a tangential homotopy equivalence this
normal invariant lies in the image of $[M,SG]_{(p)}$.  We need
to show that the image $\eta_T(f)$ of this class in
$[M,G/\mathrm{Top}]_{(p)}$ is trivial.

By Proposition \ref{prop42} we know that $\eta_T(f)$ lies in
the direct summand $E_0 K_{(p)}(M)$ of the group
$[M,G/\mathrm{Top}]_{(p)}$.
We claim that $f$ induces the identity map on this summand, and
we shall do this using the standard embedding of the summand in
$\widetilde{K}_{(p)}(M)$.
Suppose now that the homomorphism $f_*$ induced by $f$ on
$\pi_1(M)\cong\Z_p$ is multiplication by $u$, where $u\in\Z$
is prime to $p$.

The $K$-groups of $\Z_p$ lens spaces were computed by T. Kambe in
\cite{Kam}; for
our purposes, one important aspect of the computations is
that the cannoical map from $R(\Z_p)$ to $K(M)$, which sends a
representation $V$ to the vector bundle $\widetilde{M}\times_{\Z_p}
V$, is onto.  Using this
and the condition $f_*(y)=u\,y$ in the preceding paragraph,
one can prove directly that the induced automorphism $f^*$ of
$\widetilde{K}_{(p)}(M)$ is given by the Adams operation
$\psi^u$.  The restriction of this operation to the direct summand
$E_0 K_{(p)}(M)$ is the identity, and therefore we have
shown that $f^*\eta_T(f)=\eta_T(f)$.  Of course, we also have
a similar equation in which $f^*$ is replaced by its inverse.

Now let $\varphi$ be an arbitrary homotopy self-equivalence
of $M$.  Then the standard formaula for the normal invariant
of a composite implies that
$$\eta_T(\varphi\tinycirc f)~~=~~
\eta_T(\varphi)~+~(f^*)^{-1}\eta_T(f)~~=~~
\eta_T(\varphi)~+~\eta_T(f)$$
where the first equation is true by general considerations
(see the displayed formula on page 143 of \cite{inertia})
and the second follows by the last sentence of the preceding
paragraph.  Therefore, if $h_j$ is $j$-fold composite of $f$
with itself, then we have $\eta_T(h_j)=j\,\eta_T(f)$.  Consider
the special case $j=p-1$.  By hypothesis, the automorphism of fundamental
groups induced by $h_{p-1}$ is multiplication by $u^{p-1}$, which
is congruent to 1 mod $p$, and hence $h_{p-1}$ induces the
identity on $\pi_1(M)\cong\Z_p$.  This means that the associated map of
universal covering spaces ${\widetilde{h_{p-1}}}$ is an equivariant
homotopy self-equivalence of the sphere ${\widetilde{M}}$, and since
the degree of an equivariant self-map of the latter is congruent to
1 mod $p$ it follows that the degree of ${\widetilde{h_{p-1}}}$ must
be equal to 1.  Since equivariant self-maps of spheres are classified
up to equivariant homotopy by their degrees, this means that
${\widetilde{h_{p-1}}}$ is equivariantly homotopic to the identity,
which in turn implies that $h_{p-1}$ is homotopic to the identity.

The preceding sentence implies that $0=\eta_T(h_{p-1})$, and therefore
$\eta_T(h_{p-1})= (p-1)\eta_T(f)$ implies that the expression on
the right hand side of the latter equation is zero.  Since
$[M,G/\mathrm{Top}]_{(p)}$ is a finite abelian
$p$-group, it follows that $\eta_T(f)$ must also be equal to zero.
\end{proof}

\centerline{\it Tangential thickness and lens spaces}

One can also consider further specializations of the tangential
thickness question.  In particular, the case if homotopically
equivalent (genuine) lens spaces touches on interesting points in
several directions that we shall describe briefly.

It was shown in \cite{KwaSch3}  that if
$M$ and $N$ are linear space forms such that $M\times\mathbb{R}^2$
is homeomorphic to $N\times\mathbb{R}^2$, then $M$ and $N$ are
diffeomorphic.  On the other hand, we have the following result:

\begin{prop}
Let $f: M^{2n-1}\rightarrow N^{2n-1}$ be a tangential homotopy
equivalence of lens spaces whose fundamental groups have order
$p$, where $p$ is an odd prime and $n\leq p-1$.  Then
$M\times\R^3$ and $N\times\R^3$ are diffeomorphic.
\end{prop}

\textbf{Remark.} Techniques of S. Cappell and J. Shaneson in
\cite{CaSh} imply that a result analogous to Theorem 1 remains
true for $\Z_{2^r}$ lens spaces.

Proposition 5.3 follows because the normal invariant lies in
the image of the homomorphism
\[
[M,SG]_{(p)}~\longrightarrow~[M,G/O]_{(p)}
\]
and this mapping is trivial; the triviality of the normal
invariant and the $\pi-\pi$ Theorem then
imply that $f\times\mathrm{Id}(\R^3)$ is homotopic to a
diffeomorphism.

\textbf{Examples. } A closer examination of results due to J. Ewing, S.
Moolgavkar, R. Stong and L. Smith \cite{EMSS} shows that for each
$n\geq 2$ there are infinitely many primes $p$ for which one
has homotopy equivalent but not diffeomorphic lens spaces that
are stably parallelizable.  Thus for each $n\geq 2$ there are many
examples of nonhomeomorphic lens spaces $M^{2n-1}$ and $N^{2n-1}$
such that $M\times\R^3$ and $N\times\R^3$ are diffeomorphic.

It is natural to ask whether Proposition 5.3 extends to higher
values of $n$, and this question is directly
related to a remarkable theorem of J. Folkman \cite{Fo}:

\textbf{Theorem (Folkman):}  \textit{Let $p$ be an odd prime, and
assume that $n\geq 2p+1$. If
two $2n-1$-dimensional lens spaces with fundamental group
$\mathbb{Z}_p$ have the same tangential homotopy type, then they must
actually be isometric (diffeomorphic).}

One immediate question is whether the
tangential homotopy self-equivalences in the preceding theorem
are (homotopic to) diffeomorphisms.  Note that the 3-dimensional
examples preceding Proposition 5.2 lie below the dimension range
in which Folkman's result applies;  it seems highly plausible that
there are also tangential homotopy equivalences in higher dimensions
which are not homotopic to homeomorphisms.


\section{Desuspension Results}

If $X$ is a connected finite complex, it is well known that the
standard ``loop sum with identity" bijection from $\{X,S^0\}$ to
$[X,SG]$ is not necessarily a homomorphism with respect to the
loop sum structure on the domain and the composition/direct sum
structure on the codomain; specifically, if we view $\{X,S^0\}$ as
a ring using the standard smash product ring spectrum structure on
the spectrum for $S^0$, then the composition/direct sum structure
is given by:
$$
\mathrm{``}a~\oplus~b\mathrm{''}~~=~~
a~+~b~+~a\,\wedge\,b~~=~~a\,\tinycirc\, b
$$
(\textit{i.e.}, the Perlis circle operationfor the stable
cohomotopy ring; see \cite{Kap},
p. 81, line 4, or \cite{PolSeh}, Section 9.4, p. 298).
Fortunately, one can often show that these two
algebraic structures are similar in key respects (for example,
they are equal if $X$ is a suspension \cite{Whi}, pp. 124-125).
In particular, we have the following:

\begin{prop}
Let $r$ be an arbitrary positive integer, and let $\sigma:
\{B\Z_p,S^0\}\rightarrow [B\Z_p,SG]$ be the standard set-theoretic
isomorphism such that $\sigma^{-1}(u\oplus v) = u+v+uv$.  Then
$y\in \{B\Z_p,S^0\}$ is divisible by $p^r$ with respect to the
loop sum operation if and only if $\sigma{y}$ is divisible by
$p^r$ with respect to the composition or direct sum (or circle)
operation.
\end{prop}

\begin{proof}
By construction the standard map from $\Omega^{\infty}_0
S^{\infty}$ to $SG$ induces a set-theoretic bijection from
$\{B\Z_p,S^0\}$ to $[B\Z_p,SG]$ which is skeletal filtration
preserving.  The sets in these skeletal filtrations are subgroups
with respect to the standard binary operations on the respective
sets.  Therefore the sets $\mathbf{XF}_k(\{B\Z_p,S^0\})$ are
subgroups with respect to both the loop sum and the circle
operation corresponding to the operation on $[B\Z_p,SG]$.
Furthermore, it follows that each subquotient
$\mathbf{FF}_k(\{B\Z_p,S^0\})$ has group structures given by each
binary operation.  These subquotients have order equal to 1 or
$p$; since $\{B\Z_p,S^0\}$ and $[B\Z_p,SG]$ are both isomorphic to
$\widehat{\Z_{(p)}}$,  this means that the classes in
$\mathbf{FF}_k(\{B\Z_p,S^0\})$ are precisely those which are
divisible by the same prime power $p^t$ with respect to each
operation.
\end{proof}

We shall need the following dualization of Proposition 6.1 for lens
spaces:

\begin{prop}
Let $T(L)\subseteq \{L,S^0\}$ denote the image of $\{B\Z_p,S^0\}$
in $\{L,S^0\}$, so that $T(L)$ corresponds to a cyclic subgroup
$[L,J_p]$ of order $p^m$ in $[L,SG]$, where
$$
m=\bigg[\frac{n}{p-1}\bigg].
$$
Then $T(L)$ is a cyclic subgroup of $\{L,S^0\}$ with respect to
the loop sum, and for all positive integers $t$, a class $x\in
T(L)$ has order $p^t$ with respect to the loop sum if and only if
it has order $p^t$ with respect to the circle operation.
\end{prop}

\begin{proof}
The assertion that $T(L)$ is a finite cyclic group follows because
the image of $\{B\Z_p,S^0\}$ is a subgroup with respect to the
loop sum, the group $\{L,S^0\}$ is finite, and a finite quotient
of $\widehat{Z_{(p)}}\cong \{B\Z_p,S^0\}$ must be cyclic.  As
before the sets in the skeletal filtration are subgroups with
respect to both binary operations, and the subquotients either
have order 1 or $p$.  Since the set of all elements of exponent
$p^t$ in $\Z_{p^m}$ is cyclic of order $p^t$, it follows that
there is some $k$ such that $\mathbf{XF}_k(\{B\Z_p,S^0\})$ has
order $p^t$ and the latter contains all elements of exponent $p^t$
with respect to both operations.  Likewise, there is some $k' > k$
such that $\mathbf{XF}_{k'}(\{B\Z_p,S^0\})$ has order $p^{t-1}$
with respect to either operation.  Therefore, $\mathbf{XF}_k
\smallsetminus \mathbf{XF}_{k'}$ is the set of all elements with
order $p^t$ for each operation.
\end{proof}

In view of the results from Section 3, we are interested in
determining how far one can desuspend the classes in $T(L)$, and
here is the main result:

\begin{prop}
Let $k$ be an integer such that $1\leq k\leq m-1$, where
$L$, $n$ and $m$ are given as above,  Then a class in $T(L)$
desuspends to $[S^{2k-1}L,S^{2k-1}]$ if and only if its order
divides $p^k$.
\end{prop}

\begin{proof}
Fundamental results of F. Cohen, J.C. Moore and J. Neisendorfer
\cite{CohMoNei1}, \cite{CohMoNei2} imply that if a $p$-primary
element $\alpha$ in the stable homotopy groups of spheres
desuspends to $\pi_{m+2k+1}(S^{2k+1})$, then the orders of the
element $\alpha$ and its preimage have orders dividing $p^k$.  In
fact these methods immediately yield a far more general
conclusion:

\begin{lem}
Let $X$ be a finite complex, and let $\alpha$ be a $p$-primary
element of the stable cohomotopy group $\{X,S^0\}_{(p)}$ which
desuspends to $[S^{2k+1}X,S^{2k+1}]_{(p)}$.  Then $\alpha$ and its
preimage have orders dividing $p^k$.
\end{lem}

The ``only if" part of Proposition 6.3 is an immediate consequence
of this result.

\begin{proof}{\sf (Lemma 6.1)}\quad
As noted in \cite{Nei}, Cor. 11.8.2, p. 461, if $\Psi_p: \Omega^2
S^{2r+1}\rightarrow \Omega^2 S^{2r+1}$ is the double looping of
the degree $p$ self-map for $S^{2r+1}$, then $\Psi_p =
\sigma\tinycirc\pi$, where $\sigma: S^{2r-1}\rightarrow \Omega^2
S^{2r+1}$ is adjoint to the identity and $\pi:\Omega^2
S^{2r+1}\to S^{2r-1}$ is a map defined in \cite{Nei} whose precise
description is not needeed here. Therefore, if $Y$ is a
connected finite complex, then the $H$-space structure on $S^{2r+1}$
and the square lemma (see \cite{hilton-duality}, Theorem 1.5, p. 5)
imply that the degree $p$ map from $S^{2r+1}$ from itself induces
multiplication by $p$ on $[S^{2r+1}Y,S^{2r+1}]_{(p)}$ with respect to
the usual addition defned on $[A,B]$ when $A$ is a suspension,
and if $\beta\in [S^{2r+1}Y,S^{2r+1}]_{(p)}$, then
$p\cdot\beta$ desuspends to $[S^{2r-1}Y,S^{2r-1}]_{(p)}$.  One can
now proceed by induction as in \cite{CohMoNei1} and \cite{Nei} to
conclude that $p^r\cdot\beta = 0$ (\textit{e.g.}, see the proof of
\cite{Nei}, Cor. 11.8.3, p. 462).
\end{proof}

\noindent
\textit{Proof of necessity in Proposition $6.3$.}  By Lemma 6.1 and
Proposition 6.2, it will suffice to show that a generator $\tau$
of $T(L)$ desuspends to $S^{2t+1}$, where
$t=\left[\frac{n}{p-1}\right]$. Since $\tau$ has order $p^t$ by
Proposition 6.2, we can use Lemma 6.4 to conclude that $\tau$
cannot double desuspend any further. Similarly, if $r<t$, then it
will follow that $p^r\tau$ must desuspend to $S^{2(t-r)+1}$ but
cannot double desuspend any further.  The conclusion in
Proposition 6.3 follows because a multiple $a\tau$ of $\tau$
satisfies $p^k(a\tau)=0$ if and only if $a$ is divisible by
$p^{t-k}$.

It is well known that the localized stabilization maps
\[
\begin{CD}
S^{2m+1}_{(p)} @>>> Q_0(S^{2m+1})~~ =~~ \lim_{k\to\infty}~\Omega^{k}_0
S^{k+2m+1}_{(p)}
\end{CD}
\]
are very highly connected.  In fact, Propositions 1.5.7 and 1.5.8
of \cite{Rav1} imply that the localized stabilization map is
$(2(m+1)(p-1)-3)$-connected (see also the paragraph
following the statement of Proposition 1.5.8 in
\cite{Rav1}). Therefore, if $2n-1\leq 2(m+1)(p-1)-3$, then
$\tau$ (and its loop sum multiples) will automatically desuspend
to $S^{2m+1}$. In particular, if the preceding inequality holds
when $m=\left[\frac{n}{p-1}\right]$, then $\tau$ will desuspend to
$S^{2m+1}$, and hence the conclusion of Proposition 6.3 will
follow.

Write $n=j(p-1)+s$, where $0\leq s\leq p-1$, so that
$j=\left[\frac{n}{p-1}\right]$.  With this notation the dimension
versus connectivity inequality reduces to
$$
m~~\geq~~ \frac{n+1}{p-1} -1~~ =~~ j+\frac{s+1}{p-1} -1
$$
and, as indicated in the preceding paragraph, we want to verify
that this holds when $m=j$.  To see this, note that $0\leq
s\leq p-2$ implies
$$
-1~~<~~\frac{s+1}{p-1} -1~~\leq~~\frac{p-1}{p-1} -1~~=~~ 0
$$
and therefore we do have
$$
j~~\geq~~j + \frac{s+1}{p-1} -1
$$
which is what we wanted to verify.
\end{proof}

As at the beginning of this section, let $X$ be a connected finite
complex.  The remarks in the first paragraph of this section show
that, if we take the loop sum operation on $\{X,S^0\}$ and the
direct sum operation on $[X,G/\mathrm{Top}]$, then the composite
\[
\begin{CD}
\{X,S^0\} @>>> [X,SG] @>>> [X,G/\mathrm{Top}]
\end{CD}
\]
is not usually additive.  However, we have the following useful
result:

\begin{prop}
In the setting above, there is an infinite loop space structure on
$G/\mathrm{Top}$ such that the displayed composite is a
homomorphism.
\end{prop}

In fact, this structure is given by suitable versions of D.
Sullivan's Characteristic Variety Theorem (compare \cite{Sul1},
\cite{Jo} or \cite{Nic}).

\begin{proof}
If $X$ is a closed oriented manifold, then the infinite loop space
structure on $G/\mathrm{Top}$ has the following description on the
the set $[X,G/\mathrm{Top}]$: $(i)$ Take Sullivan's family of morphisms
$\varphi_i: V_i\rightarrow X$, where each $V_i$ is either a closed
manifold or a near-manifold with explicitly specified singularities.
$(ii)$ For each $\alpha\in [X,G/\mathrm{Top}]$ construct
surgery problems associated to the various classes $\varphi^*_i
\alpha\in [X,G/\mathrm{Top}]$, and take their Kervaire invariant
or (possibly reduced) signature invariants which live in suitable
cyclic abelian groups $\Lambda_i$.  These yield an embedding of
$[X,G/\mathrm{Top}]$ into $\prod_i\Lambda_i$, and the abelian
group operation on $[X,G/\mathrm{Top}]$ given by this embedding
corresponds to the Characteristic Variety infinite loop space
structure on $G/\mathrm{Top}$ (the associated spectrum is
frequently denoted by symbols like $\mathbb{L}(1)$).

Suppose now that we are given classes $u$ and $v$ in $\{X,S^0\}$,
and let $\chi(u),\chi(v)\in\prod_i\Lambda_i$ by given by the
Characteristic Variety construction.  We need to show that
$\chi(u+v) = \chi(u) + \chi(v)$.  One way of constructing
tangential surgery problems associated to $u$ and $v$ is to begin
by taking their $S$-duals, which lie in the stable homotopy groups
$\pi_{\dim X}^\textbf{S}(X^{\nu})$, where as usual $X^{\nu}$
denotes
the Thom complex of the (formally) 0-dimensional stable normal
bundle $\nu$ of $X$.  If we make these dual maps ``transverse to
the zero section" (stably of course), we obtain degree zero
tangential normal maps $(f_i,b_i)$ for suitable $f_i:
Y_i\rightarrow X$ ($i=u,v$).  The surgery problems associated to
$u$, $v$ and $u+v$ are then given by
$$
(f_u,b_u)~\amalg~\mathrm{Id}_X,\qquad (f_v,b_v)~\amalg~\mathrm{Id}_X,\qquad
(f_u,b_u)~\amalg~(f_v,b_v)~\amalg~\mathrm{Id}_X
$$
respectively.  It is now straightforward to check that if
$\chi(u)$ and $\chi(v)$ are the characteristic variety surgery
obstructions for $u$ and $v$ respectively, then $\chi(u)+\chi(v)$
will give the characteristic variety surgery obstructions for $u+v$
(see \cite{Ran} for a more detailed analysis of such problems).
This proves the result when $X$ is a closed manifold.

To verify the additivity property when $X$ is an arbitrary finite
complex, first note that $X$ has the homotopy type of a compact
manifold with boundary $W$ and $W$ is a retract of $V=\partial
(W\times[0,1])$.  The retraction $V\to W$ defines natural 1--1
mappings in homotopy $[W,Y]\to [V,Y]$ and $\{W,Y\}\to \{V,Y\}$,
where $Y$ is one of the codomains in the sequence
\[
\begin{CD}
\{Z,S^0\} @>>> [Z,SG] @>>> [Z,G/\mathrm{Top}]
\end{CD}
\]
and $Z=V$ or $W$.  Therefore the additivity properties for the
unbounded manifold
$Z=V$ imply the corresponding additivity properties for the
bounded manifold $Z=W$.
\end{proof}

The preceding result yields a useful complement to Proposition
6.3.

\begin{prop}
Let $p$ be an odd prime, let $X$ be a closed oriented manifold,
and let $a\in[X,G/\mathrm{Top}]_{(p)}$ be a class which lies in
the image of
\[
\begin{CD}
[S^{2k+1}X,S^{2k+1}]_{(p)} @>>> \{X,S^0\}_{(p)}~~\cong~~ [X,SG]_{(p)}
@>>> [X,G/\mathrm{Top}]_{(p)}
\end{CD}
\]
where $k\geq 1$.  If $*$ denotes the binary operation on the
codomain given by the Characteristic Variety Theorem and $*^py$
denotes $y*y*\cdots *y$ ($p$ factors), then $*^pa$ lies in the
image of $[S^{2k-1}X,S^{2k-1}]_{(p)}$.
\end{prop}

\textbf{Remark. }  The results of \cite{AdPri} imply that the
direct sum and Characteristic Variety structures determine
isomorphic group structures on $[X,G/\mathrm{Top}]_{(p)}$, but
the self-map inducing this isomorphism is not necessarily the
identity map.
\begin{proof}
Let $a'\in[S^{2k+1}X,S^{2k+1}]_{(p)}$ be a preimage of $a$.  Then
if $*$ denotes the loop sum in $[S^{2k+1}X,S^{2k+1}]_{(p)}$, the
results of Cohen, Moore and Neisendorfer imply that $*^pa'$ lies
in the image of $[S^{2k-1}X,S^{2k-1}]_{(p)}$.  Since the displayed
composite is additive if we take the loop space sum on the domain
and the Characteristic Variety sum on the codomains, it follows
that $*^pa$ lifts in the described manner.
\end{proof}

In the next section we shall prove a similar result if the
Characteristic Variety operation is replaced by the direct sum and
$X$ is a mod $p$ lens space (see Proposition 7.2).

\section{Proofs of Theorems 4--7}

As in Sections 3--6, unless stated otherwise, we take $p$ to be a
fixed odd prime.

All that remains is to combine the results of Sections 3--6 into
proofs of the results on $\TT_k^{\mathrm{Top}}(L)$, where
$k\geq 3$ and $L$ is a $\Z_p$ lens space.  In fact, since the
orbit space of an arbitrary free $\Z_p$ action on a sphere is
homotopy equivalent to a lens space, one can extend the entire
discussion to cases where $L$ is a fake lens space.  We begin with
the result (Theorem 5) characterizing the normal invariants of
homotopy structures in $\TT_k^{\mathrm{Top}}(L)$ for $k\geq
3$.

\begin{proof}{\sf (Theorem 5)}\quad
By Proposition 3.2 the set $\theta_k([M,G/\mathrm{Top}])$ consists
of all classes which are in the image of the normal invariant map
$\eta$ and in the image of the map
\[
\begin{CD}
[M,SG_k]@>>> [M,G/\mathrm{Top}]
\end{CD}
\]
We are assuming that the image of $\eta$ is a subgroup, so the proof
of the first part reduces to
checking that the image of $[M,SG_k]$ in
$[M,G/\mathrm{Top}]$ is a subgroup.  The composition product
defines a group structure on $[M,SG_k]$, and the stabilization map
from the latter to $[M,SG]$ is a homomorphism with respect to the
composition operation of $SG$.  But the composition and direct sum
operations are identical on the set of homotopy classes $[M,SG]$,
and since $[M,SG]\rightarrow [M,G/\mathrm{Top}]$ is a homomorphism
with respect to connected sum, it follows that the image of
$[M,SG_k]$ in $[M,G/\mathrm{Top}]$ is a subgroup.  The second part
of Theorem 5 follows by combining Proposition 3.2 with the
conclusions in the final sentence of Proposition 3.1.
\end{proof}

We can now prove Theorem 4 very easily.

\begin{proof}{\sf (Theorem 4)}\quad
Suppose that $L_0$ is homotopy equivalent to a $\Z_p$ lens space $L_0$.
Then  $[L_0,G/\mathrm{Top}]$ is a cyclic $p$-group (hence of odd
order), and Proposition 3.5 implies that the image of $[L_0,SG_3]$
in $[L_0,G/O]$, and hence also in $[L_0,G/\mathrm{Top}]$, must also be
trivial.  But this means that $\theta_3([L_0,G/\mathrm{Top}]) =0$
and therefore a homotopy structure in $\TT_3^{\mathrm{Top}}(L_0)$
must be normally cobordant to the identity.

It remains to prove that classes in $\TT_3^{\mathrm{Top}}(L_0)
\smallsetminus
\TT_2^{\mathrm{Top}}(L_0)$ are detected by the Atiyah-Singer invariant.
In order to avoid notational conflicts we shall assume that the
fundamental group is isomorphic to $\Z_q$ where $q$ is an odd prime
(so we are using $q$ instead of $p$).
The preceding argument shows that every element in $\TT_3^{\mathrm{Top}}
(L_0)$ comes from the action of the Wall group $L^h_{2k}(\Z_q)$ on
the identity mapping from $L_0$ to itself.  This group is isomorphic to
$$L^{s,p}_{2k}(\Z_q)~\oplus~{\widetilde{H}}^0\left(\Z_2;{\widetilde{K_0}}
(\Z[\Z_q])\,\right)$$
where the first summand is isomorphic to
$L_0(\{1\})\oplus\Z^{(q-1)/2}$ (compare \cite{Bak}, p. 388).  The
action of $\Z^{(q-1)/2}$ on $\mathrm{id}:L_0\to L_0$ is detected by
the Atiyah-Singer invariant, and the proof that the latter detects
elements of $\TT_3^{\mathrm{Top}}(L_0)\smallsetminus
\TT_2^{\mathrm{Top}}(L_0)$ amounts to checking that the image of
the summand ${\widetilde{H}}^0(\Z_2;{\widetilde{K_0}}(\Z[\Z_q])$
lies in the subset $\TT_2^{\mathrm{Top}}(L_0)$.  Since this assertion was
already verified in the proof of Theorem 3, it follows that the
Atiyah-Singer invariant detects the difference between
$\TT_3^{\mathrm{Top}}(L_0)$ and $\TT_2^{\mathrm{Top}}(L_0)$.
\end{proof}

Our next result implies the conclusions of Theorems 6--8 for
$\theta_{2k+1}([L,G/\mathrm{Top}])$ where $k\geq 2$.

\begin{prop}
For all $k\geq 2 $ the subquotients
$\theta_{2k}([L,G/\mathrm{Top}])/\theta_{2k-2}([L,G/\mathrm{Top}])$
are either trivial or cyclic of order $p$.  Furthermore, either
$\theta_{2k-1}([L,G/\mathrm{Top}]) =
\theta_{2k-2}([L,G/\mathrm{Top}])$ or
$\theta_{2k-1}([L,G/\mathrm{Top}]) =
\theta_{2k}([L,G/\mathrm{Top}])$
\end{prop}

Since $\theta_3([L,G/\mathrm{Top}]) = 0$ by Theorem 4, we set
$\theta_2([L,G/\mathrm{Top}]) = 0$ by definition.

\begin{proof}
Suppose that $x\in[L,G/\mathrm{Top}]_{(p)}$ lies in the image of
$[L,SG_{2k}]_{(p)}$.  Then by Proposition 3.4 we know that $x$ also
lies in the image of $[L,SF_{2k-1}]_{(p)}\cong
[S^{2k-1}L,S^{2k-1}]_{(p)}$.  Therefore, if we let $\star(k,w)$
denote the $k$-fold loop or Characteristic Variety sum on
$[L,SG]_{(p)}$ or $[L,G/\mathrm{Top}]_{(p)}$, then by Proposition
6.5 we know that the loop sum $\star(p,x)$ lies in the image of
$[S^{2k-3}L,S^{2k-3}]_{(p)}$.  We need to show this implies that
$px$ (the $p$-fold loop or composition sum of $x$ with itself)
lies in the image of $[L,SG_{2k-2}]_{(p)}$.

If $\tau$ is the generator of the cyclic $p$-group $T(L)$
described in Section 6, then there are unique integers $r\geq
0$ and $b$ such that $b$ is prime to $p$ and
$\star(bp^r,\tau)\in \{L,S^0\}_{(p)}$ maps to $x$.  By
Proposition 6.4 we know that $\star(bp^{r+1},\tau)$ maps to
$\star(p,x)$.

Proposition 6.1 now implies that
$\star(bp^r,\tau) = b'p^r\tau$ and $\star(bp^{r+1},x) =
b''p^{r+1}\tau$ for some integers $b'$ and $b''$ prime to $p$.  By
construction and our previous observations, it follows that
$b'p^r\tau$ maps to $x\in\theta_{2k}([L,G/\mathrm{Top}])$ and
$b''p^{r+1}\tau$ maps to some element of
$\theta_{2k-2}([L,G/\mathrm{Top}])$.  Choose an integer $c$ such
that $cb''\equiv b'$ modulo a sufficiently large power of $p$ (say
at least $p^n$).  Then we can also conclude that
$$
p\,x = p\cdot\mathrm{Image}(b'p^r\tau) = \mathrm{Image}(b'p^{r+1}\tau) =
\mathrm{Image}(cb''p^{r+1}\tau) = c\cdot\mathrm{Image}(b''p^{r+1}\tau)
$$
lies in $\theta_{2k-2}([L,G/\mathrm{Top}])$.  Since $x$ is
arbitrary, this means that
$$
p\ \theta_{2k}([L,G/\mathrm{Top}])~~\subseteq~~
\theta_{2k-2}([L,G/\mathrm{Top}])
$$
and since the image of $[L,SG]_{(p)}$ in $[L,G/\mathrm{Top}]$ is a
finite cyclic $p$-group this means that
$\theta_{2k}([L,G/\mathrm{Top}])/\theta_{2k-2}([L,G/\mathrm{Top}])$
is either trivial or cyclic of order $p$.
\end{proof}

\textbf{Note.} If $X$ is a finite complex, it is well known that
an element of $\{X,S^0\}_{(p)}$ desuspends to
$[S^{2k}X,S^{2k}]_{(p)}$ if and only if it desuspends to
$[S^{2k-1}X,S^{2k-1}]_{(p)}$ (\textit{e.g.}, see \cite{Rav1} or
\cite{Tod}) but apparently very little is known about
classes in $[X,SG]_{(p)}\cong \{X,S^0\}_{(p)}$ which lift to
$[X,SG_{2k+1}]_{(p)}$ outside the stable range where
$[S^{2k-1}X,S^{2k-1}]_{(p)}\rightarrow \{X,S^0\}_{(p)}$ is an
isomorphism (as in Section 6, this is roughly the range in which
$(p-1)k\geq \dim X$.

\begin{proof}{\sf (Theorem 6)}
The stable range results of Section 6 imply that
$\theta_{2k}([L,G/\mathrm{Top}]) = \theta_N([L,G/\mathrm{Top}])$
for all $N\geq 2k$ if we take $k=\left[\frac{n}{p-1}\right]$.
Therefore, by the preceding result, it is only necessary to prove
that
$$
\theta_{2k}([L,G/\mathrm{Top}])/\theta_{2k-2}
([L,G/\mathrm{Top}])~~\cong~~\Z_p
$$
if $2\leq k\leq\left[\frac{n}{p-1}\right] +1$.
 Since $n\not\equiv 0$ mod $p-1$, we know that $[L,J_p]$ maps
bijectively onto the image of $[L,SG_{(p)}]$ in
$[L,G/\mathrm{Top}]$ and that $[L,J_p]$ is cyclic of order
$p^{\left[\frac{n}{p-1}\right]}$.  By Proposition 7.1 we know that
the subquotients
$\theta_{2k}([L,G/\mathrm{Top}])/\theta_{2k-2}([L,G/\mathrm{Top}])$
have order equal to 1 or $p$, where $2\leq k\leq
\left[\frac{n}{p-1}\right] +1$.  Since the product of their orders
equals the order of $[L,J_p]$, it follows that each factor
$\theta_{2k}([L,G/\mathrm{Top}])/\theta_{2k-2}([L,G/\mathrm{Top}])$
must have order $p$.
\end{proof}

\begin{proof}{\sf (Theorem 7)}
The main difference between this case and the previous ones is
that the map from $[L,J_p]$ to $[L,G/\mathrm{Top}]$ has a kernel
isomorphic to the nonzero group $\pi_{2n-1}(J_p)$.  Similarly, if
$T'(L)$ is the image of $T(L)$ in $[L,G/\mathrm{Top}]$ (with
$T(L)$ as in Section 6), then the map $T(L)\rightarrow T'(L)$ has
a kernel of order $p$.  We now have $\left[\frac{n}{p-1}\right]$
factors of the form
$\theta_{2k}([L,G/\mathrm{Top}])/\theta_{2k-2}([L,G/\mathrm{Top}])$
where $2\leq k\leq \left[\frac{n}{p-1}\right] +1$, but the
order of $T'(L)$ is $p^{\left[\frac{n}{p-1}\right] -1}$. Since the
orders of the factors are again either 1 or $p$ and their product
is the order of $T'(L)$, it follows that all but one factor
$\theta_{2k}([L,G/\mathrm{Top}])/\theta_{2k-2}([L,G/\mathrm{Top}])$
must have order $p$ and the remaining factor will necessarily have
order 1.
\end{proof}

In general, the determination of the exceptional factor
$\theta_{2k}([L,G/\mathrm{Top}])/\theta_{2k-2}([L,G/\mathrm{Top}])$
seems to be a very difficult problem in homotopy theory.  However, one can
obtain strong restrictions on $k$ for a smooth version of the
tangential thickness problem, and these lead to partial results in
other cases.  We shall only illustrate the latter with a few
examples; their statement requires the following observation.

\begin{prop}
In the setting as above, assume that $\dim L = 2n-1$ where
$n=p^\nu(p-1)r$ for some $\nu\geq 1$ and $r$ is prime to $p$.
Let $J\theta_n(L)$
denote all classes in $[L,G/O]_{(p)}$ which lie in the images of
$[L,J_p]$ and $[L,SG_n]_{(p)}$.  Then for all $k$ the quotients
$J\theta_{2k}(L)/J\theta_{2k-2}(L)$ are either $0$ or $\Z_p$.  The
quotients vanish if $k > \left[\frac{n}{p-1}\right] +1$, and
there is also a
unique value $k_0$ of $k$ such that $2\leq k_0\leq \nu+1\leq
\left[\frac{n}{p-1}\right] +1$ and the quotient vanishes.
\end{prop}

\begin{proof}
Let $y\in [X,SG_{(p)}]$, where $X$ is a connected complex.  Then
the image of $y$ in $[X,G/O]_{(p)}$ lifts to $[X,SG_{2k(p)}]$ if
and only if $y=y_1+y_2$ where $y_1$ lies in the latter group and
$y_2$ lies in the image of $[X,SO]_{(p)}\rightarrow [X,SG]_{(p)}$.
If $X$ is the lens space $L^{2n-1}$, then by Proposition 4.5 we
know that $[L,J_p]$ is generated by the image of $\pi_{2n-1}(J_p)$
under the degree 1 normal map from $L^{2n-1}$ to $S^{2n-1}$.
Furthermore, the latter map also induces an isomorphism from
$\pi_{2n-1}(SO)$ to $[L,SO]$, and hence the image of $[L,SO]$ in
$[L,J_p]$ equals the image of $\pi_{2n-1}(J_p)$ in $[L,J_p]$.  By
naturality, the images of $[B\Z_p,J_p]\cong [B\Z_p,SG]$ and
$\pi_{2n-1}(J_p) = (\hbox{Image of the $J$-homomorphism in }
\pi_{2n-1}(SG))$ are also subgroups with respect to the loop sum
operation on $SG_{(p)}$ that we have denoted by $*$ or $\star$,
and therefore it turns out that $[L,J_p]\subseteq [L,SG]_{(p)}$ is
also a subgroup with respect to this loop space operation (of
course, usually one cannot expect such a conclusion).

Since $[L,J_p]$ contains the image of $J$, a class $y\in [L,J_p]$
maps to $J\theta_{2k}(L)\subseteq [L,G/O]_{(p)}$ if and only if it
has the form $y_1*y_2$ (with respect to the loop sum), where $y_2$
comes from $\pi_{2n-1}(J_p)$ and $y_1$ lies in the image of the
stabilization map from $[S^{2k-1}L,S^{2k-1}]_{(p)}$ to
$\{L,S^0\}_{(p)}$.  As noted before, the group $\pi_{2n-1}(J_p)$
is cyclic of order $p^\nu$.  Therefore, if a class $w\in T(L) =
\mathrm{Image}\;[B\Z_p,J_p]$ has order $p^j$ for $j\geq \nu
+1$, then $w$ plus anything coming from $\pi_{2n-1}(J_p)$ will
also have order $p^j$.  In particular, this means that no such sum
can desuspend to $[S^{2j-1}L,S^{2j-1}]_{(p)}$.  On the other hand,
it is known that the generator of $\pi_{2n-1}(J_p)$ does desuspend
to $\pi_{(2n-1)+(2\nu+1)}(S^{2\nu+1})$ (for example, see
\cite{CraKna}), and if we combine this with Proposition 6.3 we
conclude that if $k\geq \nu +1$, then a class lies in
$J\theta_{2k}(L)$ if and only if it has order dividing $p^k$.  The
nontriviality assertion about the quotients
$J\theta_{2k+2}(L)/J\theta_{2k}(L)$ is an immediate consequence of
this.
\end{proof}

When $\nu = 1$ the proposition states that the factors
$J\theta_{2k+2}(L)/J\theta_{2k}(L)$ are nontrivial for all
$k\geq 2 = \nu+1$, and by the arguments employed in the
proofs of Theorems 6 and 7 it follows that
$J\theta_4(L)/J\theta_2(L)$ must be trivial.  This suggests the
following:

\textbf{Conjecture.}  \textit{In Theorem $7$, the unique trivial
quotient
$$\theta_{2k}([L,G/\mathrm{Top}]_{(p)})/
\theta_{2k-2}([L,G/\mathrm{Top}]_{(p)})$$
is given by
$\theta_4([L,G/\mathrm{Top}]_{(p)})/\theta_2([L,G/\mathrm{Top}]_{(p)})$
and accordingly the remaining quotients\\
$\theta_{2k}([L,G/\mathrm{Top}]_{(p)})/\theta_{2k-2}([L,G/\mathrm{Top}]_{(p)
})$
are nontrivial for all $k\geq 3$.}

Finally, we shall use Proposition 7.2 to verify the conjecture on
exceptional dimensions when $n=j(p-1)$ for $j=1,\ldots ,p-1$.

\begin{prop}
Assume the setting of Theorem $7$ and Proposition $7.2$, and also let
$n=j(p-1)$ where $1\leq j\leq p-1$.  Then the quotients
$\theta_{2k+2}([L,G/\mathrm{Top}])/\theta_{2k}([L,G/\mathrm{Top}])$
are isomorphic to $\Z_p$ if $k\geq 2$, and
$\theta_4([L,G/\mathrm{Top}]) = 0$.
\end{prop}

\begin{proof}
The space $\CokJ_p$ is $(2p(p-1)-3)$-connected (see \cite{Tod}),
and therefore $[L,J_p]\cong [L,SG_{(p)}]$ if $2n-1 < 2p(p-1)-3$.
As in the proof of Proposition 7.2, a class $y\in [L,J_p]$ maps
into $\theta_{2j}([L,G/\mathrm{Top}])$ if and only if $y=y_1+y_2$
where $y_1$ comes from $[L,\mathrm{STop}]_{(p)}$ and $y_2$
desuspends to $[S^{2k-1}L,S^{2k-1}]_{(p)}$.

In order to proceed further, we need to examine the image of
$[L,\mathrm{STop}]_{(p)}$ in $[L,SG]_{(p)}$ using the Sullivan
splittings:
$$
\mathrm{STop}_{(p)}\simeq SO_{(p)}\times \CokJ_p~,\qquad
SG_{(p)}\simeq J_p\times \CokJ_p
$$

It follows that the image of $[L,\mathrm{STop}]_{(p)}$ in
$[L,SG]_{(p)}$ is the sum of $[L,\CokJ_p]$ with the image of
$[L,SO]_{(p)}$ in $[L,J_p]$.  If we are in the connectivity range
of $\CokJ_p$, this means that the images of $[L,SO]_{(p)}$ and
$[L,\mathrm{STop}]_{(p)}$ in $[L,SG]_{(p)}\cong [L,J_p]$ are equal.
Therefore, if $n$ satisfies the constraint in the proposition then
we have $J\theta_{2k}(L) = \theta_{2k}([L,G/\mathrm{Top}])$. Since
the quotients $J\theta_{2k+2}(L)/J\theta_{2k}(L)$ satisfy the
conditions in the proposition, it follows that the quotients
$\theta_{2k+2}([L,G/\mathrm{Top}])/\theta_{2k}([L,G/\mathrm{Top}])$
also satisfy these if $n=j(p-1)$ where $1\leq j\leq
p-1$.
\end{proof}

It should be possible to extend the range of dimensions for which
similar conclusions hold if one uses
the splitting of the suspension of $B\Z_p$ in \cite{Hol}
and known results on $\pi_m(\CokJ_p)$
for $m$ roughly less than $2j(p-1)$ where $j< p^2 -
$({\sf some constant}) due to Toda \cite{Tod}, but the calculations
needed to do this would be considerably more complicated than the
ones in this paper.

\section{Comments on the Smooth Case and Twisted Tangential Thickness}

In this section we shall discuss two variants of topological
tangential thickness which arise naturally in other contexts.

\subsection{Smooth Tangential Thickness and Lens Spaces}

It is clearly possible to introduce a corresponding notion of
tangential thickness in the smooth category, and in fact this
goes all the way back to \cite{Maz}.  We shall state one
such result for fake lens spaces without proof:

\begin{prop}
Let $L^{2n-1}$ be a lens space, where $n\geq 2$.
\begin{itemize}
\item[(i)] If $n\not\equiv 0$ mod $p-1$, then for each $k$ such
that $1\leq k\leq \left[\frac{n}{p-1}\right]$ there is a
manifold $L_k$ (which is tangentially homotopy equivalent to $L$)
such that $L_k\times\R^{2k}$ and $L\times\R^{2k}$ are not
homeomorphic but $L_k\times\R^{2k+2}$ and $L\times\R^{2k+2}$ are
diffeomorphic.

\item[(ii)] If $n=p^\nu(p-1)r$ where $\nu\geq 0$ and $r$ is
prime to $p$, then the same conclusion holds for all but one
value of $k$, and the exceptional value is less than or equal to
$\nu+1$.
\end{itemize}
\end{prop}

One easy way of seeing the relative complexity of smooth
tangential thickness is to consider this question for products of
the form $(L\# \Sigma^{2n-1})\times\R^k$, where $\Sigma^{2n-1}$ is
an exotic sphere.  If the order of $\Sigma^{2n-1}$ in the
Kervaire-Milnor group $\Theta_{2n-1}$ is prime to $p$ and
$k\geq 3$, then fairly standard considerations show that
$(L\#\Sigma^{2n-1})\times\R^k$ is
diffeomorphic to $L\times\R^k$ if and only if
$\Sigma^{2n-1}\times\R^k$ and $S^{2n-1}\times\R^k$ are
diffeomorphic.  Smooth tangential thickness for exotic spheres has
been fairly well understood for more than four decades (compare
\cite{Sch}; several individuals discovered these results
independently). In particular, if we combine these results with some
homotopy-theoretic input, we have the following:

\begin{prop}\

\begin{itemize}
\item[(i)] Suppose that $2n-1 = 2^j+1\geq 17$.  Then these is
a homotopy sphere $\Sigma^{2n-1}$ such that
$\Sigma^{2n-1}\times\R^{2n-7}$ is not diffeomorphic to
$S^{2n-1}\times\R^{2n-7}$ but their products with $\R$ are
diffeomorphic.

\item[(ii)] Suppose $2n-1 = 8k+1\geq 9$, and let
$\Sigma^{8k+1}$ be a homotopy sphere not bounding a spin manifold.
Then $\Sigma^{8k+1}\times\R^3$ and $S^{2k+1}\times\R^3$ are not
diffeomorphic,  but one can choose $\Sigma^{8k+1}$ such that their
products with $\R$ are diffeomorphic.
\end{itemize}
\end{prop}

The first statement follows by choosing $\Sigma$ so that its
Pontrjagin-Thom invariant of $\Sigma$ in the group
$\pi_{2n-1}/\mathrm{Image}\;J$ is
the $\eta_j\eta$, where $\eta_j\in \pi_{2^j}$ is the
Mahowald element (see \cite{Rav1}, Thm. 1.5.27(a), p. 38) and
$\eta\in\pi_1$ is the stabilization of the Hopf map from $S^3$
to $S^2$.
The proofs that the products are diffeomorphic or not diffeomorphic
are based
upon results from \cite{Mah2}, which show that $\eta_j\eta$
desuspends to $S^{2^j-4}$ but does not
desuspend to $S^{2^j-5}$({\it e.g.\/}, see Table 1 on pp. 74--75 and
Table 4.4 on p. 11;
in the setting of the previous citation from \cite{Rav1}, the class
$\eta_j\eta$ corresponds to
$$\nu^2~~\in~~E_1^{2^j+1,2^j-5}~~\approx~~\pi_6~~\approx~~\Z_2$$
in the 2-primary $EHP$ spectral sequence described in \cite{Rav1}).
The second statement follows from the fact
that the elements $\mu_k\in\pi_{8k+1}$ desuspend to $S^3$ but not
$S^2$ (see \cite{KwaSch1} for desuspension to $S^3$; as in Section 3,
if $\mu_k$ desuspended to $S^2$ it would be divisible by $\eta$ in
$\pi_*$).  As noted above, these yield diffeomorphism and
nondiffeomorphism results for smooth and fake lens spaces.

\subsection{Twisted Tangential Thickness}

The previously mentioned theorem of B. Mazur (see \cite{Maz}) has
a natural generalization to vector bundles (\textit{i.e.,} Theorem
$2$ in \cite{Maz}).

\begin{thm} {\rm (Mazur)}\quad
Let $E$ and $F$ be the total spaces of $\R^k$ bundles over smooth
closed manifolds $M^n$ and $N^n$ with $k\geq n+2$.  Then $E$
and $F$ are tangentially homotopy equivalent if and only if $E$
and $F$ are diffeomorphic.
\end{thm}

This leads to an obvious notion of \textbf{twisted tangential
thickness} which in turn has nontrivial application to the geometry
of nonnegatively curved manifolds ({\textit{cf.}} \cite{BKS}).
Namely, let $\Theta_7\cong\Z_{28}$ be the group of homotopy 7-spheres
$\Sigma^7(d)$, $d=0,\ldots ,27$.  It is proved in \cite{BKS} that
although $\Sigma^7(d)\times\mathbb{CP}^2$ falls into one
tangential homotopy type, there are four oriented (three
unoriented) diffeomorphism classes of these manifolds, each
admitting a nonnegatively curved metric by the main result of
\cite{GroZi}. Similar results hold for manifolds of the form
$\Sigma^7\times\mathbb{CP}^{2n}$ for all $n$ such that $n\not\equiv
0$ mod 3 \cite{Sch2}.

It turns out that these four different manifolds have twisted
tangential thickness 2; \textit{i.e.}, the corresponding total
spaces of the nontrivial $\R^2$ bundle are all diffeomorphic.
This result gives the first examples of manifolds with complete
enumeration of the different souls for Riemannian metrics which
admit metrics of nonnegative sectional curvature:

\noindent
\textbf{Theorem.} (Thm. 6 in \cite{BKS}) \textit{The total space
$N$ of any nontrivial complex line bundle over
$S^7\times\mathbb{CP}^2$ admits three complete nonnegatively
curved metrics with pairwise nondiffeomorphic souls --- $S_0$, $S_1$
and $S_2$ --- such that for any complete nonnegatively curved metric
on $N$ with soul $S$, there exists a self-diffeomorphism of $N$
taking $S$ to some $S_i$.}

\textbf{Remarks.} {\bf 1. }\quad
Previously M. \"Ozayd{\i}n and G. Walschap described examples
of vector bundle total spaces which support no complete metrics with
nonnegative sectional curvature \cite{OW}.

{\bf 2. }\quad It is worthwhile to note that the twisted
tangential thickness of the non-diffeomorphic manifolds
$\Sigma^7(d)\times\mathbb{CP}^2$ is equal to 2, but the standard
(untwisted) tangential thickness is equal to 3.

\end{document}